\documentclass[12pt, reqno]{amsart}
\makeatletter
\@namedef{subjclassname@1991}{$\mathrm{1991}$ Mathematics Subject Classification}
\@namedef{subjclassname@2000}{$\mathrm{2000}$ Mathematics Subject Classification}
\@namedef{subjclassname@2010}{$\mathrm{2010}$ Mathematics Subject Classification}
\@namedef{subjclassname@2020}{$\mathrm{2020}$ Mathematics Subject Classification}
\makeatother
\usepackage{amsmath,amsthm, amscd, amsfonts, amssymb, graphicx, color}
\usepackage[bookmarksnumbered, colorlinks, plainpages,linkcolor=blue,urlcolor=blue,citecolor=blue]{hyperref}
\newtheorem{thm}{Theorem}[section]
\newtheorem{cor}[thm]{Corollary}
\newtheorem{lem}[thm]{Lemma}

\newtheorem{defn}[thm]{Definition}
\newtheorem{rem}[thm]{\bf{Remark}}

\newtheorem{exam}[thm]{Example}
\numberwithin{equation}{section}

\begin{document}

\title{generalized right group inverse in Banach *-algebras}

\author{Huanyin Chen}
\author{Marjan Sheibani}
\address{School of Big Data, Fuzhou University of International Studies and Trade, Fuzhou 350202, China}
\email{<huanyinchenfz@163.com>}
\address{Farzanegan Campus, Semnan University, Semnan, Iran}
\email{<m.sheibani@semnan.ac.ir>}

\subjclass[2020]{16U90, 15A09, 16W10.}\keywords{right group inverse; generalized right Drazin inverse; generalized right core-EP inverse;
right core inverse; Banach algebra.}

\begin{abstract} In this paper, we introduce the concept of the generalized right group inverse within the context of a *-Banach algebra. This represents a natural extension of the generalized (weak) group inverse. Notably, this generalized inverse is characterized by integrating the right group inverse with the concept of quasinilpotency. We provide various characterizations and representations of the generalized right group inverse. Furthermore, we explore the relationship between the generalized right group inverse and the generalized right EP-inverse. The properties of the generalized (weak) group inverse in a Banach *-algebra are also extended to a more general framework.\end{abstract}

\maketitle

\section{Introduction}

A Banach algebra is called a Banach *-algebra if there exists an involution $*: x\to x^*$ satisfying $(x+y)^*=x^*+y^*, (\lambda x)^*=\overline{\lambda} x^*, (xy)^*=y^*x^*, (x^*)^*=x$. The involution $*$ is proper if $x^*x=0\Longrightarrow x=0$ for any $x\in \mathcal{A}$. Let ${\Bbb C}^{n\times n}$ be the Banach algebra of all $n\times n$ complex matrices, with conjugate transpose $*$ as the involution. Then the involution $*$ is proper. Throughout the paper,
$\mathcal{A}$ denotes a Banach *-algebra with proper involution.

An element $a\in \mathcal{A}$ has a group inverse if there exists an
$x\in \mathcal{A}$ such that $$xa^2=a, ax^2=x, ax=xa.$$
Such an $x$ , if it exists, is unique and is denoted by
$a^{\#}$, termed the group inverse of $a$. Clearly, a square complex matrix
$A$ possesses a group inverse if and only if $rank(A)=rank(A^2)$. Group inverse plays an important role in matrix and operator theory (see~\cite{B2,M1,M2}).

An element $a\in \mathcal{A}$ has a weak group inverse there exist $x\in R$ and $n\in \Bbb{N}$ such that
$$ax^2=x, (a^*a^2x)^*=a^*a^2x, xa^{n+1}=a^n.$$ If such $x$ exists, it is unique, and denote it by $a^{\tiny\textcircled{W}}$. Evidently, a square complex matrix
$A$ possesses a group inverse if and only if the system of equations $$AX^2=X, AX=A^{\tiny\textcircled{\dag}}A$$ is solvable.
Here, $A^{\tiny\textcircled{\dag}}$ stands for the core-EP inverse of $A$ (see~\cite{G,M1,M2,W1}).
For further reading on the weak group inverse, we refer the reader to the following papers ~\cite{F,D,M3,M4,M5,W2,Y1,Z1,Z2,Z3,Z4}.

Following Chen and Sheibani (see~\cite{C3}), an element $a\in \mathcal{A}$ has generalized group inverse if there exists $x\in \mathcal{A}$ such that $$ax^2=x, (a^*a^2x)^*=a^*a^2x, \lim\limits_{n\to \infty}||a^n-xa^{n+1}||^{\frac{1}{n}}=0.$$ Such $x$ is unique if exists, denoted by $a^{\tiny\textcircled{g}}.$ Many properties of generalized group inverse were investigated in ~\cite{C1,C2,C3}. For a complex matrix, the generalized group inverse and weak group inverse coincide with each other.

Recently, various one-sided generalized inverses are introduced and studied. Following Yan (see~\cite{Y2}), an element $a\in \mathcal{A}$ has right group inverse provided that there exists $x\in \mathcal{A}$ such that $$ax^2=x, a^2x=axa=a.$$ We use $\mathcal{A}_r^{\#}$ to denote the set of all right group invertible elements in $\mathcal{A}$.

An $a\in \mathcal{A}$ has generalized right Drazin inverse if there exists $x\in \mathcal{A}$ such that $$ax^2=x, a^2x=axa, a-axa\in \mathcal{A}^{qnil}.$$ Here, $$\mathcal{A}^{qnil}=\{x\in \mathcal{A}~\mid~ \lim\limits_{n\to \infty}\parallel x^n\parallel^{\frac{1}{n}}=0\}.$$ As is well known, $x\in \mathcal{A}^{qnil}$ if and only if $1+\lambda x\in \mathcal{A}$ is invertible. We use $\mathcal{A}_r^{d}$ to denote the set of all generalized right Drazin
invertible elements in $\mathcal{A}$. Many characterizations of generalized right (left) Drazin inverse are established in ~\cite{MB,Y2}.

The motivation of this paper is to introduce and study a new kind of generalized inverse as a natural generalization of the generalized inverses mentioned above.

In Section 2, we introduce the concept of a generalized right group inverse based on a novel form of generalized right group decomposition. Interestingly, this generalized inverse can be characterized by integrating the right group inverse with the notion of quasinilpotency.

\begin{defn} An element $a\in \mathcal{A}$ has generalized right group decomposition if there exist $x,y\in \mathcal{A}$ such that $$a=x+y, x^*y=yx=0, x\in
\mathcal{A}_r^{\#}, y\in \mathcal{A}^{qnil}.$$\end{defn}

We prove that $a\in \mathcal{A}$ has generalized right group inverse if there exists a $x\in \mathcal{A}$ such that $$x=ax^2, (a^*a^2x)^*=a^*a^2x, \lim\limits_{n\to \infty}||a^n-axa^{n}||^{\frac{1}{n}}=0.$$ The preceding $x$ is called the generalized right group inverse of $a$, and denoted by $a_r^{\tiny\textcircled{g}}$.

In Section 3, we present some characterizations and representations of generalized right group inverse. We prove that $a\in \mathcal{A}$ has generalized right group inverse if and only if $a\in \mathcal{A}_r^d$ and there exist $x\in \mathcal{A}$ such that $(aa_r^d)^*(aa_r^d)x=(aa_r^d)^*a.$
In this case, $a_r^{\tiny\textcircled{g}}=(a_r^d)^2aa_r^dx.$

An element $a\in \mathcal{A}$ has generalized right core-EP inverse if there exists $x\in \mathcal{A}$ such that $$ax^2=x, (ax)^*=ax, \lim\limits_{n\to
\infty}||a^n-axa^{n}||^{\frac{1}{n}}=0.$$ If such a $x$ exists, we denote it by $a_r^{\tiny\textcircled{d}}$.

Finally, in Section 4, we discuss the relationship between the generalized right group inverse and the generalized right EP-inverse. Consequently, the properties of the weak (generalized) group inverse for Hilbert space operators are expanded to a more general framework (refer to~\cite{C3,M3,M5}).

\section{generalized right group inverse}

The purpose of this section is to introduce a new generalized inverse which is a natural generalization of generalized (weak) group inverse in a Banach *-algebra. We begin with

\begin{lem} Let $a\in \mathcal{A}_r^{\#}$. Then $a\in \mathcal{A}_r^{d}$.\end{lem}
\begin{proof} Straightforward.\end{proof}

\begin{lem} Let $a\in \mathcal{A}$. Then $a\in \mathcal{A}_r^{\#}$ if and only if there exists $x\in \mathcal{A}$ such that
 $$ax^2=x=xax, a^2x=axa, a=xa^2.$$\end{lem}
\begin{proof} One direction is trivial. Conversely, assume that there exists $x\in \mathcal{A}$ such that
 $$ax^2=x, a^2x=axa=a.$$ Set $z=xax$. Then we check that
 $$\begin{array}{rll}
 az&=&axax=ax,\\
 az^2&=&ax(xax)=(ax^2)ax=xax=z,\\
 zaz&=&x(axa)xax=x(axa)x=xax=z,\\
 a^2z&=&a^2x=a,\\
 aza&=&a(xax)a=(axa)xa=axa=a.
 \end{array}$$ Therefore $a\in \mathcal{A}_r^{\#}$, as desired.\end{proof}

The preceding $x$ in Lemma 2.2 is called the right group inverse of $a$, we denote it by $a_r^{\#}$. We use
$\{ a_r^{\#}\}$ to stand the set of all right group inverses of $a$.

\begin{lem} Let $a\in \mathcal{A}$. Then $a\in \mathcal{A}_r^{d}$ if and only if there exists $x\in \mathcal{A}$ such that
$$ax^2=x=xax, a^2x=axa, a-axa\in \mathcal{A}^{qnil}.$$\end{lem}
\begin{proof} One direction is obvious. Conversely, assume that there exists $x\in \mathcal{A}$ such that
 $$ax^2=x, a^2x=axa, a-axa\in \mathcal{A}^{qnil}.$$ Then $(a-axa)axa=a^2xa-axa^2xa=a^2xa-a^3x^2a=a^2xa-a^2(ax^2)a=0$, and so
 $(a-axa)^2=(a-axa)a-(a-axa)axa=(a-axa)a$. This implies that $$\begin{array}{rll}
 (a-axa)^n&=&(a-axa)^{n-2}(a-axa)^2=(a-axa)^{n-2}(a-axa)a\\
 &=&(a-axa)^{n-3}(a-axa)^2a=(a-axa)^{n-2}a^2\\
 &=&\cdots =(a-axa)a^{n-1}=a^n-axa^n.
 \end{array}$$ Since $a-axa\in \mathcal{A}^{qnil}$, we have $\lim\limits_{n\to \infty}||(a-axa)^n||^{\frac{1}{n}}=0.$ Hence, $\lim\limits_{n\to \infty}||a^n-axa^n||^{\frac{1}{n}}=0.$ We directly verify that
 $$\begin{array}{rll}
 ax-(ax)^2&=&ax-axax\\
 &=&a^nx^n-axa^nx^n=(a^n-axa^n)x^n\\
 &=&(a-axa)^nx^n
 \end{array}$$ Then $$\lim\limits_{n\to \infty}||ax-(ax)^2||^{\frac{1}{n}}=0.$$
 This implies that $(ax)^2=ax$. Set $z=xax$. Then we verify that $az^2=z=zaz, a^2z=aza$. Moreover, we have
 $a-aza=a-a(xax)a=a-axa\in \mathcal{A}^{qnil}$, as asserted.\end{proof}

The preceding $x$ in Lemma 2.3 is called the generalized right Drazin inverse of $a$, we denote it by $a_r^{d}$. Set $$\{ a_r^d\}=\{ x\in \mathcal{A}~|~ax^2=x=xax, a^2x=axa, a-xa^2\in \mathcal{A}^{qnil}\}.$$ Evidently,
$a\in \mathcal{A}_r^d$ if and only if $\{ a_r^d\}\neq \emptyset$.

\begin{lem} Let $a\in \mathcal{A}_r^{d}$ and $b\in \mathcal{A}^{qnil}$. If $ba=0$, then $a+b\in \mathcal{A}_r^{d}$. In this case,
$$(a+b)_r^d=a_r^d[1+\sum\limits_{n=1}^{\infty}(a_r^d)^nb^n].$$\end{lem}
\begin{proof} Let $x\in \{ a_r^d\}$ and $y=x[1+\sum\limits_{n=1}^{\infty}x^nb^n]$. Since $ba=0$, we see that $bx=(ba)x^2=0$. Then we verify that
$$\begin{array}{rll}
(a+b)y^2&=&ay^2=(ay)y\\
&=&ax[1+\sum\limits_{n=1}^{\infty}x^nb^n]x[1+\sum\limits_{n=1}^{\infty}x^nb^n]\\
&=&(ax^2)x[1+\sum\limits_{n=1}^{\infty}x^nb^n]=y,\\
y(a+b)y&=&yay=x[1+\sum\limits_{n=1}^{\infty}x^nb^n]ax[1+\sum\limits_{n=1}^{\infty}x^nb^n]\\
&=&(xax)[1+\sum\limits_{n=1}^{\infty}x^nb^n]=x[1+\sum\limits_{n=1}^{\infty}x^nb^n]=y,\\
(a+b)^2y&=&(a+b)ay=a^2y=(a^2x)[1+\sum\limits_{n=1}^{\infty}x^nb^n]\\
&=&(axa)[1+\sum\limits_{n=1}^{\infty}x^nb^n]=axa+(axa)\sum\limits_{n=1}^{\infty}x^nb^n\\
&=&axa+(a^2x)\sum\limits_{n=1}^{\infty}x^nb^n=axa+axb+\sum\limits_{n=2}^{\infty}x^{n-1}b^n\\
&=&axa+axb+\sum\limits_{n=1}^{\infty}x^{n}b^{n+1}\\
&=&axa+axb+ax\sum\limits_{n=1}^{\infty}x^nb^{n+1}]\\
&=&axa+ax[b+\sum\limits_{n=1}^{\infty}x^nb^{n+1}]\\
&=&aya+ayb=ay(a+b)=(a+b)y(a+b),\\
\end{array}$$
$$\begin{array}{rll}
(a+b)-(a+b)^2y&=&(a+b)-(a^2+ab+b^2)x[1+\sum\limits_{n=1}^{\infty}x^nb^n]\\
&=&(a+b)-a^2x[1+\sum\limits_{n=1}^{\infty}x^nb^n]\\\\
&=&(a-a^2x)+[1-a^2x\sum\limits_{n=1}^{\infty}x^nb^{n-1}]b.
\end{array}$$

Since $b[1-a^2x\sum\limits_{n=1}^{\infty}x^nb^{n-1}]=b\in \mathcal{A}^{qnil}$,
by using Cline's formula (see~\cite[Theorem 2.2]{L2}), $[1-a^2x\sum\limits_{n=1}^{\infty}x^nb^{n-1}]b\in \mathcal{A}^{qnil}$.
As $[1-a^2x\sum\limits_{n=1}^{\infty}x^nb^{n-1}]b(a-a^2x)=0$ and $a-a^2x\in \mathcal{A}^{qnil}$, it follows by~\cite[Lemma 2.4]{CK} that
$(a-a^2x)+[1-a^2x\sum\limits_{n=1}^{\infty}x^nb^{n-1}]b\in \mathcal{A}^{qnil}$.
That is, $(a+b)-y(a+b)^2\in \mathcal{A}^{qnil}$. This implies that $\{ (a+b)_r^d\}\neq \emptyset$, and therefore
$a+b\in \mathcal{A}_r^{d}$.\end{proof}

\begin{lem} Let $a\in \mathcal{A}_r^d$ and $z\in \{ a_r^d\}$. Then $\lim\limits_{n\to \infty}||(a^n-aza^n)^*||^{\frac{1}{n}}=0.$
\end{lem}
\begin{proof} Let $x=a-aza$. Then $x\in \mathcal{A}^{qnil}$. For any $\lambda \in {\Bbb C}$, we have $1-\overline{\lambda} x\in \mathcal{A}^{-1}$, and so
$1-\lambda x^*\in \mathcal{A}^{-1}$. Hence, $x^*\in \mathcal{A}^{qnil}$.
Obviously, we have $$\begin{array}{rll}
(a-aza)az&=&a^2z-(aza)az=a^2z-(a^2z)az=a^2z-a(aza)z\\
&=&a^2z-a(a^2z)z=a^2z-a^3z^2=a^2z-a^2(az^2)\\
&=&a^2-a^2z=0,
\end{array}$$ and so $$(a-aza)(1-az)=a-aza.$$
Then
$$\begin{array}{rll}
||(a^n-aza^n)^*||&=&||(a^{n-1})^*(a-aza)^*||\\
&=&||(a^{n-1})^*(1-az)^*(a-aza)^*||\\
&=&||(a^{n-2})^*[(a-aza)^2]^*||\\
&\vdots&\\
&=&||[(a-aza)^n]^*||=||(x^*)^n||.
\end{array}$$ Since $\lim\limits_{n\to \infty}||(x^*)^n||^{\frac{1}{n}}=0,$ we have $\lim\limits_{n\to \infty}||(a^n-aza^n)^*||^{\frac{1}{n}}=0.$\end{proof}

\begin{thm} Let $a\in \mathcal{A}$. Then the following are equivalent:\end{thm}
\begin{enumerate}
\item [(1)] $a\in \mathcal{A}$ has generalized right group decomposition, i.e., $$a=a_1+a_2, a_1^*a_2=a_2a_1=0, a_1\in \mathcal{A}_r^{\tiny\textcircled{\#}}, a_2\in \mathcal{A}^{qnil}.$$
\vspace{-.5mm}
\item [(2)] There exists $x\in \mathcal{A}$ such that $$x=ax^2, (a^*a^2x)^*=a^*a^2x, \lim\limits_{n\to \infty}||a^n-axa^{n}||^{\frac{1}{n}}=0.$$
\vspace{-.5mm}
\item [(3)] $a\in \mathcal{A}_r^{d}$ and there exist $x\in \mathcal{A}$ such that $$x=ax^2, (aa_r^d)^*a^2x=(aa_r^d)^*a, \lim\limits_{n\to \infty}||a^n-axa^{n}||^{\frac{1}{n}}=0.$$
\end{enumerate}
\begin{proof} $(1)\Rightarrow (3)$  By hypothesis, $a$ has the generalized right group decomposition $a=a_1+a_2$.
Let $x\in \{ (a_1)_r^{\#} \}$. Then $ax=(a_1+a_2)x=a_1x+(a_2a_1)x^2=a_1x$. Hence $ax^2=a_1x^2=x$.

Since $a_2a_1=0$ and $a_2\in \mathcal{A}^{qnil}$, it follows by Lemma 2.4 that
$a\in \mathcal{A}_r^d$ and
$$a_r^d=(a_1)_r^{\#}+\sum\limits_{n=1}^{\infty}((a_1)_r^{\#})^{n+1}a_2^n.$$
Hence, $$aa_r^d=a_1(a_1)_r^{\#}+\sum\limits_{n=1}^{\infty}a_1((a_1)_r^{\#})^{n+1}a_2^n.$$
Then $$\begin{array}{rll}
(aa_r^d)^*a_2&=&(a_1(a_1)_r^{\#}))^*a_2+\sum\limits_{n=1}^{\infty}[a_1((a_1)_r^{\#})^{n+1}a_2^n]^*a_2\\
&=&[(a_1)_r^{\#}]^*(a_1^*a_2)+\sum\limits_{n=1}^{\infty}[((a_1)_r^{\#})^{n+1}a_2^n]^*(a_1^*a_2)\\
&=&0;
\end{array}$$ hence, $(aa_r^d)^*a_1=(aa_r^d)^*(a_1+a_2)=(aa_r^d)^*a$.
Accordingly, $$(aa_r^d)^*a^2x=(aa_r^d)^*(a_1+a_2)(a_1+a_2)x=(aa_r^d)^*a_1^2x=(aa_r^d)^*a_1=(aa_r^d)^*a.$$

We check that $$\begin{array}{rll}
a^n-axa^n&=&(1-ax)a^n=(1-a_1x)(a_1+a_2)a^{n-1}\\
&=&(1-a_1x)a_2a^{n-1}=(1-a_1x)a_2(a_1+a_2)a^{n-2}\\
&=&(1-a_1x)(a_2)^2a^{n-2}=\cdots =(1-a_1x)(a_2)^n.
\end{array}$$ Hence,
$$||a^n-axa^{n}||^{\frac{1}{n}}\leq ||1-a_1x||^{\frac{1}{n}}||(a_2)^n||^{\frac{1}{n}}.$$ Since $a_2\in \mathcal{A}^{qnil}$, we see that
$\lim\limits_{n\to \infty}||(a_2)^n||^{\frac{1}{n}}=0.$ Hence,
$\lim\limits_{n\to \infty}||a^n-axa^n||^{\frac{1}{n}}=0,$ as required.

$(3)\Rightarrow (2)$ By hypothesis, there exists $x\in \mathcal{A}$ such that $$x=ax^2, (aa_r^d)^*a^2x=(aa_r^d)^*a, \lim\limits_{n\to \infty}||a^n-axa^n||^{\frac{1}{n}}=0.$$
Let $a_1=a^2x$ and $a_2=a-a^2x$. Then we check that
$$\begin{array}{rll}
||a_2a_1||&=&||(a-a^2x)a^2x||=||a^3x-a^2xa^2x||\\
&=&||a^{k+1}x^{k-1}-a^2xa^kx^{k-1}||\\
&=&||a[a^k-axa^k]x^{k-1}||.
\end{array}$$ Then $$\begin{array}{rll}
||a_2a_1||^{\frac{1}{k}}&\leq &||a||^{\frac{1}{k}}||a^k-axa^k||^{\frac{1}{k}}||x^{k-1}||^{\frac{1}{k}}.
\end{array}$$ Therefore $\lim\limits_{k\to \infty}||a_2a_1||^{\frac{1}{k}}=0$, and then $a_2a_1=0$.

Since $x=ax^2, (aa_r^d)^*a^2x=(aa_r^d)^*a$, we deduce that
$$\begin{array}{rl}
&||a_1^*a_2||\\
=&||(a^2x)^*(a-a^2x)||\\
=&||(a^2x)^*a-(a^2x)^*a^2x||\\
=&||(a^2x)^*a-(a^{k}x^{k-1})^*a^2x||\\
=&||(a^2x)^*a-[(a^{k}-aa_r^da^{k})x^{k-1}+aa_r^da^kx^{k-1}]^*a^2x||\\
\leq&||(a^{k}-aa_r^da^k)^*||(x^{k-1})^*||||a^2x||+||(a^2x)^*a-[aa_r^da^kx^{k-1}]^*a^2x||\\
\leq&||(a^{k}-aa_r^da^{k})^*||(x^{k-1})^*||||a^2x||+||(a^2x)^*a-[a^kx^{k-1}]^*(aa_r^d)^*a^2x||\\
=&||(a^{k}-aa_r^da^{k})^*||(x^{k-1})^*||||a^2x||+||(a^2x)^*a-[a^kx^{k-1}]^*(aa_r^d)^*a||\\
\leq&||(a^{k}-aa_r^da^{k})^*||(x^{k-1})^*||||a^2x||+||(a^2x)^*a-[aa_r^da^{k}x^{k-1}]^*a||\\
\leq&||(a^{k}-aa_r^da^{k})^*||(x^{k-1})^*||||a^2x||+||(a^kx^{k-1})^*a-[aa_r^da^{k}x^{k-1}]^*a||\\
\leq&||(a^{k}-aa_r^da^{k})^*||(x^{k-1})^*||||a^2x||+||(a^k-aa_r^da^{k})x^{k-1}]^*||||a||\\
\leq&||(a^{k}-aa_r^da^{k})^*||||(x^{k-1})^*||||a^2x||+||(a^k-aa_r^da^{k})^*||||(x^{k-1})^*||||a||.
\end{array}$$ Then $$\begin{array}{rll}
||a_1^*a_2||^{\frac{1}{k}}&\leq &||(a^{k}-aa_r^da^{k})^*||^{\frac{1}{k}}||(x^{k-1})^*||^{\frac{1}{k}}||a^2x||^{\frac{1}{k}}\\
&+&||(a^k-aa_r^da^{k})^*||^{\frac{1}{k}}||(x^{k-1})^*||^{\frac{1}{k}}||a||^{\frac{1}{k}}.
\end{array}$$ By virtue of Lemma 2.5, $\lim\limits_{k\to \infty}||a_1^*a_2||^{\frac{1}{k}}=0$; hence, $a_1^*a_2=0$.

Since $ax^2=x$, we check that $$\begin{array}{rll}
a_1x&=&a^2x^2=a(ax^2)=ax,\\
a_1xa_1&=&(ax)(a^2x)=axa^2x=a^2x=a_1,\\
a_1x^2&=&(a_1x)x=(ax)x=ax^2=x,\\
(a_1)^2x&=&a_1(a_1x)=(a^2x)(ax)=a^2xax=a(axa)x=a^2x=a_1.
\end{array}$$ By virtue of Lemma 2.2, $a_1\in \mathcal{A}_r^{\#}$.
Clearly, $a=a^2x+(a-a^2x)=a_1+a_2$. Then we have $a^*a^2x=(a_1^*+a_2^*)(a_1+a_2)^2(a_1)_r^{\#}=(a_1^*+a_2^*)a_1=a_1^*a_1+(a_1^*a_2)^*=a_1^*a_1$.
Accordingly, $(a^*a^2x)^*=(a_1^*a_1)^*=a_1^*a_1=a^*a^2x$, as required.

$(2)\Rightarrow (1)$ By hypotheses, we have $z\in \mathcal{A}$ such that $$z=az^2, (a^*a^2z)^*=a^*a^2z, \lim\limits_{n\to \infty}||a^n-aza^{n}||^{\frac{1}{n}}=0.$$
For any $n\in {\Bbb N}$, we have $az=a(az^2)=a^2z^2=a^2(az^2)z=a^3z^3=\cdots =a^nz^n=\cdots =a^{n+1}z^{n+1}$.
Thus, we prove that $$\begin{array}{rll}
||az-azaz||&=&||a^nz^n-az(a^nz^n)||\\
&=&||(a^n-aza^n)z^n||.
\end{array}$$ Hence, $$\begin{array}{rll}
||az-azaz||^{\frac{1}{n}}&\leq &||(a^n-aza^n)||^{\frac{1}{n}}||z||^{\frac{n}{n}}.
\end{array}$$
This implies that $$\lim\limits_{n\to \infty}||az-azaz||^{\frac{1}{n}}=0,$$ whence, $az=(az)^2$.

Set $x=a^2z$ and $y=a-a^2z.$ Then $a=x+y$.
We check that $$\begin{array}{rcl}
(a^2-aza^2)z&=&(a^2-aza^2)a^{n-2}z^{n-1}\\
&=&(a^n-aza^n)z^{n-1}.
\end{array}$$
Therefore $$||(a^2-aza^2)z||^{\frac{1}{n}}\leq ||a^n-aza^{n}|^{\frac{1}{n}}|||z^{n-1}||^{\frac{1}{n}}.$$
Since $$\lim\limits_{n\to \infty}||a^n-aza^n||^{\frac{1}{n}}=0,$$ we derive that $$\lim\limits_{n\to \infty}||(a^2-aza^2)z||^{\frac{1}{n}}=0.$$
Hence, $(a^2-aza^2)z=0$.

We claim that $x$ has right group inverse. Evidently, we verify that
$$\begin{array}{rll}
x^2z&=&a^2za^2z^2=a^2zaz=a(az)^2=a^2z=x,\\
xz^2&=&(a^2z)z^2=(a^2z^2)z=az^2=z,\\
xzx&=&a^2z^2a^2z=aza^2z=a^2z=x.
\end{array}$$ Hence, $x\in \mathcal{A}_r^{\#}$ by Lemma 2.2.

We check that
$$\begin{array}{rll}
||(a-aza)^{n+1}||^{\frac{1}{n+1}}&=&||(a-aza)^{n-1}(a-aza)(a-aza)||^{\frac{1}{n+1}}\\
&=&||(a-aza)^{n-1}(a-aza)a||^{\frac{1}{n+1}}\\
&=&||(a-aza)^{n-2}(a-aza)(a-aza)a||^{\frac{1}{n+1}}\\
&=&||(a-aza)^{n-2}(a-aza)a^2||^{\frac{1}{n+1}}\\
&\vdots&\\
&=&||(a-aza)(a-aza)a^{n-1}||^{\frac{1}{n+1}}\\
&\leq &\big[||a^n-aza^{n}||^{\frac{1}{n}}\big]^{\frac{n}{n+1}}||a-aza||^{\frac{1}{n+1}}.
\end{array}$$ Accordingly, $$\lim\limits_{n\to \infty}||(a-aza)^{n+1}||^{\frac{1}{n+1}}=0.$$ This implies that $a-aza\in \mathcal{A}^{qnil}$. By using Cline's
formula (see~\cite[Theorem 2.2]{L2}), $y=a-a^2z\in \mathcal{A}^{qnil}$. Moreover, we see that $$\begin{array}{rll}
x^*y&=&(a^2z)^*(a-a^2z)=[z^*(a^2)^*a](1-az)\\
&=&(a^*a^2z)^*(1-az)\\
&=&(a^*a^2z)(1-az)=0,\\
yx&=&(a-a^2z)(a^2z)=a^3z-a^2(za^2z)=a^3z-a^2(az)=0,
\end{array}$$ as required.\end{proof}

If the preceding generalized right group decomposition of $a$ exists, we say that $a$ has generalized right group inverse $(a_1)_r^{\tiny\textcircled{\#}}$, and denote it by $a_r^{\tiny\textcircled{g}}$, i.e., $a_r^{\tiny\textcircled{g}}=(a_1)_r^{\tiny\textcircled{\#}}$. Let $\mathcal{A}_r^{\tiny\textcircled{g}}$ denote the sets of all generalized right group invertible elements in $\mathcal{A}$.

\begin{cor} Let $a\in \mathcal{A}$. Then\end{cor}
\begin{enumerate}
\item [(1)] $a\in \mathcal{A}^{\#}$ if and only if $a\in \mathcal{A}_r^{\#}\bigcap \mathcal{A}^D$.
\vspace{-.5mm}
\item [(2)] $a\in \mathcal{A}^{\tiny\textcircled{g}}$ if and only if $a\in \mathcal{A}_r^{\tiny\textcircled{g}}\bigcap \mathcal{A}^d$.
\end{enumerate}
\begin{proof} $(1)$ One direction is obvious. Conversely, assume that $a\in \mathcal{A}_r^{\#}\bigcap \mathcal{A}^D$. Then
there exists $x\in \mathcal{A}$ such that $$ax^2=x=xax, a^2x=axa=a.$$ Then $a-a^Da^2=a^{n+1}x^{n}-a^Da^{n+2}x^n=(a^n-a^Da^{n+1})ax^n.$ Hence,
 $$||a-a^Da^2||^{\frac{1}{n}}\leq ||a^n-a^Da^{n+1}|^{\frac{1}{n}}|||a||^{\frac{1}{n}}|||x||.$$
Since $$\lim\limits_{n\to \infty}||a^n-a^Da^{n+1}||^{\frac{1}{n}}=0,$$ we have
$$\lim\limits_{n\to \infty}||a-a^Da^2||^{\frac{1}{n}}=0.$$
Hence, $a=a^Da^2=a^2x\in a^2\mathcal{A}\bigcap \mathcal{A}a^2$. Therefore $a\in \mathcal{A}^{\#}$.

$(2)$ This is obvious by Theorem 2.6 and~\cite[Theorem 3.1]{C3}.\end{proof}

\begin{cor} Let $a,b\in \mathcal{A}_r^{\tiny\textcircled{g}}$. If $ab=ba=a^*b=0$, then $a+b\in \mathcal{A}_r^{\tiny\textcircled{g}}$. In this case,
$$(a+b)_r^{\tiny\textcircled{g}}=a_r^{\tiny\textcircled{g}}+b_r^{\tiny\textcircled{g}}.$$\end{cor}
\begin{proof} In view of Theorem 2.6, we have decompositions:
$$\begin{array}{c}
a=x+y, x^*y=yx=0, x\in \mathcal{A}_r^{\#}, y\in \mathcal{A}^{qnil};\\
b=s+t, s^*t=ts=0, s\in \mathcal{A}_r^{\#}, t\in \mathcal{A}^{qnil}.
\end{array}$$ Explicitly, we have $x=a^2a_r^{\tiny\textcircled{g}}$ and $s=b^2b_r^{\tiny\textcircled{g}}$.
Then $a+b=(x+s)+(y+t)$. We easily verify that
$$\begin{array}{rll}
(x+s)(x_r^{\#}+s_r^{\#})(x+s)&=&x+s,\\
(x_r^{\#}+s_r^{\#})(x+s)^2&=&x+s,\\
(x+s)(x_r^{\#}+s_r^{\#})^2&=&x_r^{\#}+s_r^{\#},\\
(x_r^{\#}+s_r^{\#})(x+s)(x_r^{\#}+s_r^{\#})&=&x_r^{\#}+s_r^{\#}.
\end{array}$$
Then $x+s\in \mathcal{A}_r^{\#}$ and $(x+s)_r^{\#}=x_r^{\#}+s_r^{\#}.$

Obviously, we have $$a^2a_r^{\tiny\textcircled{g}}b=aa_r^da^2a_r^{\tiny\textcircled{g}}b.$$

By hypothesis, we have $(aa_r^d)^*a^2a_r^{\tiny\textcircled{g}}b=(aa_r^d)^*ab=0$; hence, $(aa_r^d)^*aa_r^da^2a_r^{\tiny\textcircled{g}}b=0$. As the involution is proper, we deduce that
$aa_r^da^2a_r^{\tiny\textcircled{g}}b=0$. This implies that $a^2a_r^{\tiny\textcircled{g}}b=0$. Likewise,
$b^2b_r^{\tiny\textcircled{g}})a=0.$
Hence, $yt=(a-a^2a_r^{\tiny\textcircled{g}})(b-b^2b_r^{\tiny\textcircled{g}})=
-a^2a_r^{\tiny\textcircled{g}}b(1-bb_r^{\tiny\textcircled{g}})=
0$, it follows by~\cite[Lemma 2.4]{CK} that $y+t\in \mathcal{A}^{qnil}$.

Obviously, we check that $$\begin{array}{rll}
(x+s)^*(y+t)&=&x^*y+x^*t+s^*y+s^*t=x^*t+s^*y\\
&=&(a_r^{\tiny\textcircled{g}}a)^*(a^*b)(1-bb_r^{\tiny\textcircled{g}})+(b_r^{\tiny\textcircled{g}}b)^*(b^*a)(1-aa_r^{\tiny\textcircled{g}})\\
&=&0,\\
(y+t)(x+s)&=&yx+ys+tx+ts=ys+tx\\
&=&(a-a^2a_r^{\tiny\textcircled{g}})b^2b_r^{\tiny\textcircled{g}}+(b-b^2b_r^{\tiny\textcircled{g}})a^2a_r^{\tiny\textcircled{g}}=0.
\end{array}$$ By using Theorem 2.6, $$\begin{array}{rll}
(a+b)_r^{\tiny\textcircled{g}}&=&(x+s)_r^{\#}\\
&=&x_r^{\#}+s_r^{\#}\\
&=&a_r^{\tiny\textcircled{g}}+b_r^{\tiny\textcircled{g}},
\end{array}$$ as asserted.\end{proof}

As is well known, $a\in \mathcal{A}$ has generalized Drazin inverse if and only if it is quasipolar, i.e., there exists an idempotent $p\in \mathcal{A}$ such that $a+p\in \mathcal{A}^{-1}, pa=ap\in \mathcal{A}^{qnil}$. We come now to characterize generalized right group inverse by the polar-like property.

\begin{thm} Let $a\in \mathcal{A}$. Then the following are equivalent:\end{thm}
\begin{enumerate}
\item [(1)]{\it $a\in \mathcal{A}_r^{\tiny\textcircled{g}}$.}
\item [(2)]{\it There exists an idempotent $p\in \mathcal{A}$ such that $$(1-p)a(1-p)\in [(1-p)\mathcal{A}(1-p)]_r^{-1}, (a^*ap)^*=a^*ap~\mbox{and} ~pa=pap\in \mathcal{A}^{qnil}.$$}
\item [(3)]{\it $a\in \mathcal{A}_r^{d}$ and there exists an idempotent $p\in \mathcal{A}$ such that $$(1-p)a(1-p)\in [(1-p)\mathcal{A}(1-p)]_r^{-1}, (aa_r^d)^*ap=0~\mbox{and} ~pa\in \mathcal{A}^{qnil}.$$}
\end{enumerate}
\begin{proof} $(1)\Rightarrow (2)$ Since $a\in R_r^{\tiny\textcircled{g}}$, there exist $z,y\in \mathcal{A}$ such that $$a=z+y, z^*y=yz=0, z\in \mathcal{A}_r^{\#}, y\in \mathcal{A}^{qnil}.$$ Set $x=z_r^{\#}$. Then we check that $$\begin{array}{rll}
ax&=&(z+y)z_r^{\#}=zz_r^{\#}+(yz)(zz_r^{\#})^2=zz_r^{\#},\\
ax^2&=&(ax)x=z(z_r^{\#})^2=z_r^{\#}=x,\\
a^*a^2x&=&(z+y)^*(z+y)zz_r^{\#}=(z+y)^*(z^2z_r^{\#})=z^*z,\\
(a^*a^2x)^*&=&(z^*z)^*=z^*z=a^*a^2x.
\end{array}$$ Let $p=1-zz_r^{\#}$. Then $p=p^2\in \mathcal{A}$. Furthermore, $ap=a(1-zz_r^{\#})=(z+y)(1-zz_r^{\#})=y\in \mathcal{A}^{qnil}$, and so
$pa=(1-zz_r^{\#})a=a-zz_r^{\#}a\in \mathcal{A}^{qnil}$ by Cline's formula (see~\cite[Theorem 2.2]{L2}).
Thus, $$(a^*ap)^*=(a^*a-a^*a^2x)^*=a^*a-a^*a^2x=a^*ap.$$
Since $pa(1-p)=(1-zz_r^{\#})(z+y)zz_r^{\#}=0$, we get $pa=pap$.
It is easy to verify that $$\begin{array}{rll}
(1-p)a(1-p)=zz_r^{\#}(z+y)zz_r^{\#}=zz_r^{\#}z^2z_r^{\#}=zz_r^{\#}z=z.
\end{array}$$ Obviously, $$[(1-p)a(1-p)][(1-p)z_r^{\#}(1-p)]=1-p.$$
Thus, $$(1-p)a(1-p)\in [(1-p)\mathcal{A}(1-p)]_r^{-1},$$ as required.

$(2)\Rightarrow (1)$ By hypothesis, there exists an idempotent $p\in \mathcal{A}$ such that $$(1-p)a(1-p)\in [(1-p)\mathcal{A}(1-p)]_r^{-1}, (a^*ap)^*=a^*ap~\mbox{and} ~pa=pap\in \mathcal{A}^{qnil}.$$ Set
$x=a(1-p)$ and $y=ap$. Then $a=x+y$. Since $pa=pap\in \mathcal{A}^{qnil}$, by using Cline's formula, we have $y\in \mathcal{A}^{qnil}$.
We also see that $pa(1-p)=0$, and then $yx=apa(1-p)=0$. Moreover, we have $x^*y=[a(1-p)]^*ap=(1-p^*)a^*ap=(1-p^*)(p^*a^*a)=0$.
We claim that $x\in \mathcal{A}_r^{\#}$.

Obviously, $x=a(1-p)=(1-p)a(1-p)$. Let $z=[(1-p)a(1-p)]_r^{-1}$. Then
$$\begin{array}{rll}
x^2z&=&a(1-p)a(1-p)z=a(1-p)=x,\\
xzx&=&(1-p)a(1-p)zx=(1-p)x=x,\\
xz^2&=&[(1-p)a(1-p)z]z=(1-p)z=z,\\
zxz&=&z(xz)=z(1-p)=z.
\end{array}$$ This implies that $x\in \mathcal{A}_r^{\#}$.

Then $a=x+y$ is a generalized right group decomposition of $a$. In light of Theorem 2.6, $a\in \mathcal{A}_r^{\tiny\textcircled{g}}$.

$(1)\Rightarrow (3)$ Clearly, $a\in \mathcal{A}_r^{d}$. By hypothesis, $a$ has the generalized right group decomposition $a=z+y$.
Let $x=z_r^{\#}$ and $p=1-z_r^{\#}$. As in the preceding discussion, we prove that
$$(1-p)a(1-p)\in [(1-p)\mathcal{A}(1-p)]_r^{-1}~\mbox{and} ~pa=pap\in \mathcal{A}^{qnil}.$$
As in the proof in Theorem 2.6, we prove that $(aa_r^d)^*a^2x=(aa_r^d)^*a$.
Therefore $(aa_r^d)^*ap=(aa_r^d)^*a-(aa_r^d)^*a^2x=0$, as desired.

$(3)\Rightarrow (1)$ By hypothesis, we have an idempotent $p\in \mathcal{A}$ such that $$(1-p)a(1-p)\in [(1-p)\mathcal{A}(1-p)]_r^{-1}, (aa_r^d)^*ap=0~\mbox{and} ~pa=pap\in \mathcal{A}^{qnil}.$$
Set $x=a(1-p)$ and $y=ap$. Then $a=x+y$. Analogously to the preceding discussion, we have $yx=0, x\in \mathcal{A}_r^{\#}$ and $y\in \mathcal{A}^{qnil}$.
By virtue of Lemma 2.4, we have
$$(x+y)_r^d=x_r^{\#}[1+\sum\limits_{n=1}^{\infty}(x_r^{\#})^ny^n].$$
Hence, $$\begin{array}{rll}
aa_r^dx&=&(x+y)(x+y)_r^dx\\
&=&(x+y)x_r^{\#}[1+\sum\limits_{n=1}^{\infty}(x_r^{\#})^ny^n]x\\
&=&(x+y)x_r^{\#}x=xx_r^{\#}x+yx_r^{\#}x\\
&=&xx_r^{\#}x+(yx)(x_r^{\#})^2x=x.
\end{array}$$
Thus we verify that $$\begin{array}{rll}
x^*y&=&(aa_r^{\#}x)^*ap\\
&=&x^*(aa_r^{d})^*ap\\
&=&0.
\end{array}$$
Hence $a=x+y$ is a generalized right group decomposition of $a$. According to Theorem 2.6, $a\in \mathcal{A}_r^{\tiny\textcircled{g}}$.
\end{proof}

\begin{cor} Let $a\in \mathcal{A}_r^{\tiny\textcircled{g}}$. Then there exists an idempotent $p\in \mathcal{A}$ such that
$$a+p\in \mathcal{A}_r^{-1}, (a^*ap)^*=a^*ap~\mbox{and} ~ap=pap\in \mathcal{A}^{qnil}.$$\end{cor}
\begin{proof} In view of Theorem 2.9, there exists an idempotent $p\in \mathcal{A}$ such that $$(1-p)a(1-p)\in [(1-p)\mathcal{A}(1-p)]_r^{-1}, (a^*ap)^*=a^*ap~\mbox{and} ~ap=pap\in \mathcal{A}^{qnil}.$$ Write $(1-p)a(1-p)(1-p)t(1-p)=1-p$ for a $t\in \mathcal{A}$.
Then $$\begin{array}{rll}
a+p&=&[(1-p)a+pa]+p=[(1-p)a(1-p)+p]+pa\\
&=&[(1-p)a(1-p)+p][1+((1-p)t(1-p)+p)pa]\\
&=&[(1-p)a(1-p)+p](1+pa).
\end{array}$$ Since $ap\in \mathcal{A}^{qnil}$, by using Cline's formula (see~\cite[Theorem 2.2]{L2}), $1+pa\in \mathcal{A}^{qnil}$.
Therefore $a+p\in \mathcal{A}_r^{-1}$, as asserted.\end{proof}

The following example illustrates that generalized right group inverse maybe different
from the generalized group inverse.

\begin{exam} Let $V$ be a countably generated infinite-dimensional vector
space over the complex field ${\Bbb C}$, and let $\{ x_1,x_2,x_3,\cdots \}$
be a basis of $V$. Let $\sigma: V\to V$ be a right shift
operator given by $\sigma(x_{i})=x_{i+1}$
for all $i\in \mathbb{N}$. Then $\sigma$ has generalized right group inverse, while it has no
generalized group inverse.\end{exam}
\begin{proof} Let $\tau: V\to V$ be a left shift
operator given by $\tau(x_1)=0, \tau(x_{i+1})=x_{i}$
for all $i\in \mathbb{N}$. Then $\sigma\tau=1_V$.
The right (left) shift operator can be regarded as an element in a Banach *-algebra of complex matrix, with conjugate transpose $*$ as the involution.
We verify that $$\sigma\tau^2=\tau=\tau\sigma\tau, \sigma^2\tau=\sigma\tau\sigma=\sigma.$$
Then $\sigma$ has right group inverse, and therefore it has generalized right group inverse. We claim that it has no generalized
group inverse.

Assume that $\sigma$ has generalized group inverse. It follows by~\cite[Theorem 2.2]{C3} that it has generalized Drazin inverse. Then there exists an operator $\eta$ such that
$$\sigma\eta=\eta\sigma, \eta=\eta^2\sigma, \sigma-\eta\sigma^2~\mbox{is quasinilpotent}.$$
This implies that $$\lim\limits_{n\to \infty}||(\sigma-\eta\sigma^2)^n||^{\frac{1}{n}}=0.$$
Since $\eta\sigma=\sigma\eta$ is an idempotent, we verify that
$$\begin{array}{rll}
||\sigma-\eta\sigma^2||^{\frac{1}{n}}&=&||(\sigma^n-\eta\sigma^{n+1})\tau^{n-1}||^{\frac{1}{n}}\\
&=&||(1-\eta\sigma)\sigma^{n}\tau^{n-1}||^{\frac{1}{n}}\\
&=&||(1-\eta\sigma)^n\sigma^{n}\tau^{n-1}||^{\frac{1}{n}}\\
&=&||(\sigma-\eta\sigma^2)^n\tau^{n-1}||^{\frac{1}{n}}\\
&\leq&||(\sigma-\eta\sigma^2)^n||^{\frac{1}{n}}||\tau||^{1-\frac{1}{n}}.
\end{array}$$ This implies that $$\lim\limits_{n\to \infty}||\sigma-\eta\sigma^2||^{\frac{1}{n}}=0.$$
Hence $\sigma=\eta\sigma^2$. We infer that $1_V=\sigma\tau=\eta\sigma^2\tau=\eta\sigma$.
According $\sigma$ is invertible, and then $\tau$ is invertible. Since all eigenvalues of $\tau$ are zero, it is not invertible.
This gives a contradiction. Therefore $\sigma$ has no generalized group inverse, as asserted.\end{proof}

\section{equivalent characterizations and representations}

In this section, we focus on the diverse characterizations of the generalized right group inverse. The representations of related generalized right group inverse are presented. Our discussion begins with the following point.

\begin{thm} Let $a\in \mathcal{A}$. Then $a\in \mathcal{A}_r^{\tiny\textcircled{g}}$ if and only if\end{thm}
\begin{enumerate}
\item [(1)] $a\in \mathcal{A}_r^d$;
\vspace{-.5mm}
\item [(2)] There exists an idempotent $q\in \mathcal{A}$ such that $$(aa_r^d)\mathcal{A}=q\mathcal{A}~\mbox{and} ~a^*aq=q^*a^*a.$$
\end{enumerate}
In this case, $a_r^{\tiny\textcircled{g}}=a_r^dq.$
\begin{proof} $\Longrightarrow $ In view of Theorem 2.6, $a\in \mathcal{A}_r^d$. Let $q=aa_r^{\tiny\textcircled{g}}$. Then $q^2=[aa_r^{\tiny\textcircled{g}}a]a_r^{\tiny\textcircled{g}}=aa_r^{\tiny\textcircled{g}}=q\in \mathcal{A}$, i.e., $q\in \mathcal{A}$ is an idempotent.
We check that $$a^*aq=a^*a^2a_r^{\tiny\textcircled{g}}=(a^*a^2a_r^{\tiny\textcircled{g}})^*=(a^*aq)^*=q^*a^*a.$$

Observing that $$\begin{array}{rll}
||aa_r^{\tiny\textcircled{g}}-aa_r^daa_r^{\tiny\textcircled{g}}||^{\frac{1}{n}}&=&
||a^n(a_r^{\tiny\textcircled{g}})^{n}-aa_r^da^n(a_r^{\tiny\textcircled{g}})^{n}||^{\frac{1}{n}}\\
&\leq &||a^n-aa_r^da^n||^{\frac{1}{n}}||a_r^{\tiny\textcircled{g}}||.
\end{array}$$
Since $$\lim\limits_{n\to \infty}||a^n-a^na_r^da^n||^{\frac{1}{n}}=0,$$ we derive
$$\lim\limits_{n\to \infty}||aa_r^{\tiny\textcircled{g}}-aa_r^daa_r^{\tiny\textcircled{g}}||^{\frac{1}{n}}=0.$$ Then $aa_r^{\tiny\textcircled{g}}=aa_r^daa_r^{\tiny\textcircled{g}}.$

Likewise, we verify that $aa_r^d=aa_r^{\tiny\textcircled{g}}aa_r^d.$
Therefore $q\mathcal{A}=aa_r^d\mathcal{A}$, as required.

$\Longleftarrow $ By hypothesis, there exists an idempotent $q\in \mathcal{A}$ such that $aa_r^d\mathcal{A}=q\mathcal{A}$ and $a^*aq=q^*a^*a.$
Set $x=a_r^dq$. Then $ax=aa_r^dq=q$, and so $$(a^*a^2x)^*=(a^*aax)^*=(a^*aq)^*=q^*a^*a=a^*aq=(a^*a)(ax)=a^*a^2x.$$ Moreover, we have
$$ax^2=(ax)x=qx=q(a_r^dq)=q(aa_r^d)a_r^dq=(aa_r^d)a_r^dq=a_r^dq=x.$$ We verify that
$$\begin{array}{rll}
||a^n-axa^n||&=&||[a^n-a(a_r^dq)aa_r^da^n]-[ax(a^n-aa_r^da^{n})]||\\
&\leq &||a^n-aa_r^da^{n}||+||ax||||a^{n}-aa_r^da^{n}||\\
&\leq&\big(1+||ax||\big)||a^n-aa_r^da^{n}||.
\end{array}$$ Since $$\lim\limits_{n\to \infty}||a^n-aa_r^da^n||^{\frac{1}{n}}=0,$$ we derive
$$\lim\limits_{n\to \infty}||a^n-axa^{n}||^{\frac{1}{n}}=0.$$ Therefore
$a\in \mathcal{A}_r^{\tiny\textcircled{g}}$ and
$a_r^{\tiny\textcircled{g}}=x=a_r^dq.$\end{proof}

\begin{cor} Let $a\in \mathcal{A}$. Then $a\in \mathcal{A}_r^{\tiny\textcircled{g}}$ if and only if\end{cor}
\begin{enumerate}
\item [(1)] $a\in \mathcal{A}_r^d$;
\vspace{-.5mm}
\item [(2)] There exists an idempotent $q\in \mathcal{A}$ such that $\ell(aa_r^d)=\ell(q)$ and $a^*aq=q^*a^*a.$
\end{enumerate}
In this case, $a_r^{\tiny\textcircled{g}}=a_r^dq.$
\begin{proof}  $\Longrightarrow $ In view of Theorem 3.1, $a\in \mathcal{A}_r^d$ and there exists an idempotent $q\in \mathcal{A}$ such that $$aa_r^d\mathcal{A}=q\mathcal{A}~\mbox{and} ~a^*aq=q^*a^*a.$$ We infer that $\ell(aa_r^d)=\ell(q)$, as desired.

$\Longleftarrow $ By hypothesis, there exists an idempotent $q\in \mathcal{A}$ such that $\ell(aa_r^d)=\ell(q)$ and $a^*aq=q^*a^*a.$
Clearly, $1-aa_r^d\in \ell(aa_r^d)\subseteq \ell(q)$; hence, $q=aa_r^dq$. This implies that $q\mathcal{A}\subseteq aa_r^d\mathcal{A}.$
Also we have $1-q\in \ell(q)\subseteq \ell(aa_r^d)$, and then $aa_r^d=qaa_r^d$. We infer that $aa_r^d\mathcal{A}\subseteq q\mathcal{A}$.
Therefore $aa_r^d\mathcal{A}=q\mathcal{A}$. By virtue of Theorem 3.1, $a\in \mathcal{A}_r^{\tiny\textcircled{g}}$. In this case,
$a_r^{\tiny\textcircled{g}}=a_r^dq.$\end{proof}

\begin{thm} Let $a\in \mathcal{A}$. Then $a\in \mathcal{A}_r^{\tiny\textcircled{g}}$ if and only if\end{thm}
\begin{enumerate}
\item [(1)] $a\in \mathcal{A}_r^d$;
\vspace{-.5mm}
\item [(2)] There exist $x\in \mathcal{A}$ such that $(aa_r^d)^*(aa_r^d)x=(aa_r^d)^*a.$
\end{enumerate}
In this case, $a_r^{\tiny\textcircled{g}}=(a_r^d)^2x.$
\begin{proof} $\Longrightarrow $ By virtue of Theorem 2.6, $a\in \mathcal{A}_r^{d}$ and there exists a $y\in \mathcal{A}$ such that $$y=ay^2, (aa_r^d)^*a^2y=(aa_r^d)^*a, \lim\limits_{n\to \infty}||a^n-aya^{n}||^{\frac{1}{n}}=0.$$

Choose $x=a^2y$. Then we verify that
$$\begin{array}{rl}
&||(aa_r^d)^*a-(aa_r^d)^*aa_r^dx||\\
=&||(aa_r^d)^*a-(aa_r^d)^*aa_r^da^2y||\\
&||(aa_r^d)^*a-(aa_r^d)^*aa_r^da^{k}y^{k-1}||\\
=&||[(aa_r^d)^*a-(aa_r^d)^*a^ky^{k-1}]+(aa_r^d)^*[a^k-aa_r^da^{k}]y^{k-1}||\\
=&||[(aa_r^d)^*a-(aa_r^d)^*a^2y]+(aa_r^d)^*[a^k-aa_r^da^{k}]y^{k-1}||\\
=&||(aa_r^d)^*[a^k-aa_r^da^{k}]y^{k-1}||\\
\leq&||(aa^d)^*||||a^k-aa_r^da^{k}||||y^{k-1}||.
\end{array}$$ Since $\lim\limits_{k\to \infty}||a^k-aa_r^da^{k}||^{\frac{1}{k}}=0$, we derive that
$\lim\limits_{k\to \infty}||(aa_r^d)^*a-(aa_r^d)^*aa_r^dx||^{\frac{1}{k}}=0.$ Hence $(aa_r^d)^*aa_r^dx=(aa_r^d)^*a,$ as required.

$\Longleftarrow $ By hypothesis, there exist $x\in \mathcal{A}$ such that $(aa_r^d)^*aa_r^dx=(aa_r^d)^*a$.
Then $(aa_r^d)^*aa_r^d=[(aa_r^d)^*a]a_r^d=[(aa_r^d)^*aa_r^dx]a_r^d=(aa_r^d)^*aa_r^dxa_r^d.$
As the involution $*$ is proper, we have
$aa_r^d=aa_r^dxa_r^d$. Choose $z=(a_r^d)^2x$. Then we verify that
$$\begin{array}{rll}
az^2&=&a[(a_r^d)^2aa_r^dx][(a_r^d)^2aa_r^dx]\\
&=&[a(a_r^d)^2][aa_r^dxa_r^d]a_r^daa_r^dx\\
&=&(a_r^d)^2aa_r^dx=z,\\
(aa_r^d)^*a^2z&=&(aa_r^d)^*a^2(a_r^d)^2aa_r^dx=(aa_r^d)^*aa_r^dx=(aa_r^d)^*a.\\
\end{array}$$

Obviously, $az=a(a_r^d)^2x=a_r^daa_r^dx$; and then $$(az)^2=a_r^d[aa_r^dxa_r^d]aa_r^dx=a_r^daa_r^daa_r^dx=a_r^daa_r^dx=az.$$

Since $(aa_r^d)^*aa_r^dx=(aa_r^d)^*a$, we deduce that $$(aa_r^d)^*aa_r^dxaa_r^d=(aa_r^d)^*(a^2a_r^d)=[(aa_r^d)^*aa_r^da].$$ As the involution is proper, we
have $aa_r^dxaa_r^d=aa_r^da=a^2a_r^d$.

Set $x=a^2z$ and $y=a-a^2z.$ Then $a=x+y$.

Obviously, $$a-a^2z=a-aa_r^dx=(a-aa_r^da)+aa_r^d(x-a).$$

We easily check that $$\begin{array}{rll}
[a-aa_r^da][aa_r^d(x-a)]&=&[a-a^2a_r^d][aa_r^d(x-a)]\\
&=&a[(1-aa_r^d)(aa_r^d)](x-a)\\
&=&0.
\end{array}$$

$$aa_r^d(x-a)aa_r^d=aa_r^dxaa_r^d-aa_r^da^2a_r^d=aa_r^dxaa_r^d-a^3(a_r^d)^2=0$$

By using Cline's formula, we have $(x-a)aa_r^d=(x-a)[aa_r^d]^2\in \mathcal{A}^{qnil}.$
Hence, $a-a^2z\in \mathcal{A}^{qnil}.$

We claim that $x$ has right group inverse. Evidently, we verify that
$$\begin{array}{rll}
x^2z&=&a^2za^2z^2=a^2zaz=a(az)^2=a^2z=x,\\
xz^2&=&(a^2z)z^2=(a^2z^2)z=az^2=z,\\
zxz&=&z(a^2z)z=za(az^2)=zaz=(a_r^d)^2aa_r^dx[a(a_r^d)^2]aa_r^dx\\
&=&(a_r^d)^2xaa_r^dx=z,\\
xzx&=&a^2z^2a^2z=aza^2z=a^3z^2=a^2z=x.
\end{array}$$ Hence, $x\in \mathcal{A}_r^{\#}.$

We directly check that $$\begin{array}{rll}
yx&=&(a-a^2z)(a^2z)=[a-a^2(a_r^d)^2aa_r^dx]a^2(a_r^d)^2aa_r^dx\\
&=&a^2a_r^dx-[aa_r^dxaa_rd]x\\
&=&a^2a_r^dx-a^2a_r^dx=0.
\end{array}$$ Since
$$az^2=z, (aa_r^d)^*a^2z=(aa_r^d)^*a,$$ as in the proof of Theorem 2.6,
we verify that $x^*y=0.$ In view of Theorem 2.6, $a\in \mathcal{A}_r^{\tiny\textcircled{g}}$. In this case,
$a_r^{\tiny\textcircled{g}}=x_r^{\#}=z=(a_r^d)^2aa_r^dx.$ This completes the proof.\end{proof}

\begin{cor} Let $a\in \mathcal{A}$. Then the following are equivalent:\end{cor}
\begin{enumerate}
\item [(1)] $a\in \mathcal{A}_r^{\tiny\textcircled{g}}$.
\vspace{-.5mm}
\item [(2)] $a\in \mathcal{A}_r^d$ and there exists an idempotent $p\in \mathcal{A}$ $$p\in aa_r^d\mathcal{A}~\mbox{and} ~(aa_r^d)^*ap=(aa_r^d)^*a.$$
\item [(3)] $a\in \mathcal{A}_r^d$ and there exists an idempotent $q\in \mathcal{A}$ $$aa_r^d\mathcal{A}=q\mathcal{A}~\mbox{and} ~(aa_r^d)^*aq=(aa_r^d)^*a.$$
\end{enumerate}
In this case, $a_r^{\tiny\textcircled{g}}=aa_r^dq$.
\begin{proof} $(1)\Rightarrow (3)$ In view of Theorem 3.3, $a\in \mathcal{A}_r^{d}$ and there exists an idempotent $q\in \mathcal{A}$ such that $$(aa_r^d)\mathcal{A}=q\mathcal{A}~\mbox{and} ~a^*aq=q^*a^*a.$$ Explicitly, $q=aa_r^{\tiny\textcircled{g}}$. We directly check that
$$\begin{array}{rll}
qaa_r^d&=&aa_r^{\tiny\textcircled{g}}aa_r^d=aa_r^d,\\
aa_r^dq&=&aa_r^daa_r^{\tiny\textcircled{g}}=q.
\end{array}$$ Then we verify that
$$\begin{array}{rll}
(aa_r^d)^*aq&=&(a_r^d)^*(a^*aq)=(a_r^d)^*(q^*a^*a)\\
&=&[aqa_r^d]^*a=[a(qaa_r^d)(a_r^d)]^*a=[a(aa_r^d)(a_r^d)]^*a=(aa_r^d)^*a.
\end{array}$$ Accordingly, $(aa_r^d)^*aq=(aa_r^d)^*a,$ as desired.

$(3)\Rightarrow (2)$ This is trivial.

$(2)\Rightarrow (1)$ By hypothesis, there exists an idempotent $p\in \mathcal{A}$ such that $$p\in aa_r^d\mathcal{A}~\mbox{and} ~(aa_r^d)^*ap=(aa_r^d)^*a.$$
Set $q=p+(1-p)aa_r^d$. Then we verify that
$$\begin{array}{rll}
q^2&=&[p+(1-p)aa_r^d][p+(1-p)aa_r^d]\\
&=&p+(1-p)aa_r^d(1-p)aa_r^d=q,\\
q&=&aa_r^d+p(1-aa_r^d)\in aa_r^d\mathcal{A},\\
aa_r^d&=&qaa_r^d\in q\mathcal{A},\\
aa_r^d\mathcal{A}&=&q\mathcal{A},\\
(aa_r^d)^*aq&=&(aa_r^d)^*a[aa_r^d+p(1-aa_r^d)]\\
&=&(aa_r^d)^*a^2a_r^d+(aa_r^d)^*ap(1-aa_r^d)\\
&=&(aa_r^d)^*a^2a_r^d+(aa_r^d)^*a(1-aa_r^d)=(aa_r^d)^*a.
\end{array}$$ Write $q=aa_r^dz$ with $z\in \mathcal{A}$. Choose $x=az$. Then $(aa_r^d)^*aa_r^dx= (aa_r^d)^*(aa_r^da)z$ $=(aa_r^d)^*a^2a_r^dz=(aa_r^d)^*aq=(aa_r^d)^*a$. This completes the proof by Theorem 3.3.\end{proof}

Next, we investigate the representation of the generalized right group inverse for an anti-triangular matrix.

\begin{thm} Let $\alpha=\left(
\begin{array}{cc}
a&b\\
0&c
\end{array}
\right)$ with $a,c\in \mathcal{A}_r^{\tiny\textcircled{g}}$. If $(1-aa_r^{\tiny\textcircled{g}})b=0$ and $b(1-cc_r^{\tiny\textcircled{g}})=0$, then
$\alpha\in M_2(\mathcal{A})_r^{\tiny\textcircled{g}}$ and $$\alpha_r^{\tiny\textcircled{g}}=\left(
\begin{array}{cc}
a_r^{\tiny\textcircled{g}}&-a_r^{\tiny\textcircled{g}}bc_r^{\tiny\textcircled{g}}\\
0&c_r^{\tiny\textcircled{g}}
\end{array}
\right).$$\end{thm}
\begin{proof} Set $x=\left(
\begin{array}{cc}
a_r^{\tiny\textcircled{g}}&-a_r^{\tiny\textcircled{g}}bc_r^{\tiny\textcircled{g}}\\
0&c_r^{\tiny\textcircled{g}}
\end{array}
\right).$ Then we verify that
$$\begin{array}{rll}
\alpha x&=&\left(
\begin{array}{cc}
a&b\\
0&c
\end{array}
\right)\left(
\begin{array}{cc}
a_r^{\tiny\textcircled{g}}&-a_r^{\tiny\textcircled{g}}bc_r^{\tiny\textcircled{g}}\\
0&c_r^{\tiny\textcircled{g}}
\end{array}
\right)\\
&=&\left(
\begin{array}{cc}
aa_r^{\tiny\textcircled{g}}&(1-aa_r^{\tiny\textcircled{g}})bc_r^{\tiny\textcircled{g}}\\
0&cc_r^{\tiny\textcircled{g}}
\end{array}
\right)=\left(
\begin{array}{cc}
aa_r^{\tiny\textcircled{g}}&0\\
0&cc_r^{\tiny\textcircled{g}}
\end{array}
\right),\\
\alpha x^2&=&\left(
\begin{array}{cc}
aa_r^{\tiny\textcircled{g}}&0\\
0&cc_r^{\tiny\textcircled{g}}
\end{array}
\right)\left(
\begin{array}{cc}
a_r^{\tiny\textcircled{g}}&-a_r^{\tiny\textcircled{g}}bc_r^{\tiny\textcircled{g}}\\
0&c_r^{\tiny\textcircled{g}}
\end{array}
\right)=x,\\
\alpha^*\alpha^2x&=&\left(
\begin{array}{cc}
a^*&0\\
b^*&c^*
\end{array}
\right)\left(
\begin{array}{cc}
a&b\\
0&c
\end{array}
\right)\left(
\begin{array}{cc}
aa_r^{\tiny\textcircled{g}}&0\\
0&cc_r^{\tiny\textcircled{g}}
\end{array}
\right)\\
&=&\left(
\begin{array}{cc}
a^*&0\\
b^*&c^*
\end{array}
\right)\left(
\begin{array}{cc}
a^2a_r^{\tiny\textcircled{g}}&bcc_r^{\tiny\textcircled{g}}\\
0&c^2c_r^{\tiny\textcircled{g}}
\end{array}
\right)\\
&=&\left(
\begin{array}{cc}
a^*a^2a_r^{\tiny\textcircled{g}}&a^*b\\
b^*a&b^*b+c^*c^2c_r^{\tiny\textcircled{g}}
\end{array}
\right).
\end{array}$$ This implies that $(\alpha^*\alpha^2x)^*=\alpha^*\alpha^2x$.

Obviously, we have $||a^3a_r^{\tiny\textcircled{g}}-aa_r^{\tiny\textcircled{g}})a^3a_r^{\tiny\textcircled{g}}||^{\frac{1}{n}}
\leq ||a^n-aa_r^{\tiny\textcircled{g}})a^n||^{\frac{1}{n}}||a_r^{\tiny\textcircled{g}}||^{1-\frac{2}{n}}.$ Since $\lim\limits_{n\to \infty}||a^n-aa_r^{\tiny\textcircled{g}})a^n||^{\frac{1}{n}}=0,$ we deduce that
$\lim\limits_{n\to \infty}||a^3a_r^{\tiny\textcircled{g}}-aa_r^{\tiny\textcircled{g}})a^3a_r^{\tiny\textcircled{g}}||^{\frac{1}{n}}=0.$
This implies that $a^3a_r^{\tiny\textcircled{g}}=aa_r^{\tiny\textcircled{g}})a^3a_r^{\tiny\textcircled{g}}$.
By hypothesis, we have $b=aa_r^{\tiny\textcircled{g}}b$, and so   $(1-aa_r^{\tiny\textcircled{g}})ab=(1-aa_r^{\tiny\textcircled{g}})a^3a_r^{\tiny\textcircled{g}}b=0$. Thus
we verify that $$\begin{array}{rll}
\alpha^n-\alpha x\alpha^n&=&(I_2-\alpha x)\alpha^n\\
&=&\left(
\begin{array}{cc}
1-aa_r^{\tiny\textcircled{g}}&0\\
0&1-cc_r^{\tiny\textcircled{g}}
\end{array}
\right)\left(
\begin{array}{cc}
a&b\\
0&c
\end{array}
\right)\alpha^{n-1}\\
&=&\left(
\begin{array}{cc}
(1-aa_r^{\tiny\textcircled{g}})a&0\\
0&(1-cc_r^{\tiny\textcircled{g}})c
\end{array}
\right)\alpha^{n-1}\\
&=&\left(
\begin{array}{cc}
(1-aa_r^{\tiny\textcircled{g}})a^2&(1-aa_r^{\tiny\textcircled{g}})ab\\
0&(1-cc_r^{\tiny\textcircled{g}})c^2
\end{array}
\right)\alpha^{n-2}\\
&=&\left(
\begin{array}{cc}
(1-aa_r^{\tiny\textcircled{g}})a^2&0\\
0&(1-cc_r^{\tiny\textcircled{g}})c^2
\end{array}
\right)\alpha^{n-2}\\
&\vdots&\\
&=&\left(
\begin{array}{cc}
(1-aa_r^{\tiny\textcircled{g}})a^n&0\\
0&(1-cc_r^{\tiny\textcircled{g}})c^n
\end{array}
\right).
\end{array}$$ Since $\lim\limits_{n\to \infty}||a^n-aa_r^da^{n}||^{\frac{1}{n}}=\lim\limits_{n\to \infty}||c^n-cc_r^dc^{n}||^{\frac{1}{n}}=0$, we deduce that
$\lim\limits_{n\to \infty}||\alpha^n-\alpha x\alpha^n||^{\frac{1}{n}}=0.$ Therefore
$$\alpha_r^{\tiny\textcircled{g}}=x=\left(
\begin{array}{cc}
a_r^{\tiny\textcircled{g}}&-a_r^{\tiny\textcircled{g}}bc_r^{\tiny\textcircled{g}}\\
0&c_r^{\tiny\textcircled{g}}
\end{array}
\right),$$ as asserted.\end{proof}

\begin{cor} Let $\alpha=\left(
\begin{array}{cc}
a&b\\
0&c
\end{array}
\right)$ with $a,c\in \mathcal{A}_r^{\tiny\textcircled{g}}$. If $(1-aa_r^d)b=0$ and $b(1-cc_r^{\tiny\textcircled{g}})=0$, then
$\alpha\in M_2(\mathcal{A})_r^{\tiny\textcircled{g}}$ and $$\alpha_r^{\tiny\textcircled{g}}=\left(
\begin{array}{cc}
a_r^{\tiny\textcircled{g}}&-a_r^{\tiny\textcircled{g}}bc_r^{\tiny\textcircled{g}}\\
0&c_r^{\tiny\textcircled{g}}
\end{array}
\right).$$\end{cor}
\begin{proof} Since $(1-aa_r^d)b=0$, we have that $b=aa_r^db$. Then
$aa_r^{\tiny\textcircled{g}}b=aa_r^{\tiny\textcircled{g}}aa_r^db=aa_r^{\tiny\textcircled{g}}aa_r^db=b$.
Therefore $(1-aa_r^{\tiny\textcircled{g}})b=0$. Therefore we obtain the result by Theorem 3.5.\end{proof}

\begin{cor} Let $\alpha=\left(
\begin{array}{cc}
a&0\\
b&c
\end{array}
\right)$ with $a,c\in \mathcal{A}_r^{\tiny\textcircled{g}}$. If $(1-cc^{\tiny\textcircled{g}})b=0$ and $b(1-aa^{\tiny\textcircled{g}})=0$, then
$\alpha\in M_2(\mathcal{A})_r^{\tiny\textcircled{g}}$ and $$\alpha_r^{\tiny\textcircled{g}}=\left(
\begin{array}{cc}
a_r^{\tiny\textcircled{g}}&0\\
-c_r^{\tiny\textcircled{g}}ba_r^{\tiny\textcircled{g}}&c_r^{\tiny\textcircled{g}}
\end{array}
\right).$$\end{cor}
\begin{proof} By virtue of Theorem 3.5, $\beta:=\left(
\begin{array}{cc}
c&b\\
0&a
\end{array}
\right)\in M_2(\mathcal{A})_r^{\tiny\textcircled{g}}$ and $$\beta_r^{\tiny\textcircled{g}}=\left(
\begin{array}{cc}
c_r^{\tiny\textcircled{g}}&-c_r^{\tiny\textcircled{g}}ba_r^{\tiny\textcircled{g}}\\
0&a_r^{\tiny\textcircled{g}}
\end{array}
\right).$$ We directly check that $$\alpha=\left(
\begin{array}{cc}
0&1\\
1&0
\end{array}
\right)\beta \left(
\begin{array}{cc}
0&1\\
1&0
\end{array}
\right), \left(
\begin{array}{cc}
0&1\\
1&0
\end{array}
\right)^2=I_2, \left(
\begin{array}{cc}
0&1\\
1&0
\end{array}
\right)^*=\left(
\begin{array}{cc}
0&1\\
1&0
\end{array}
\right).$$ Therefore $$\alpha_r^{\tiny\textcircled{g}}=\left(
\begin{array}{cc}
0&1\\
1&0
\end{array}
\right)\beta\left(
\begin{array}{cc}
0&1\\
1&0
\end{array}
\right)=\left(
\begin{array}{cc}
a_r^{\tiny\textcircled{g}}&0\\
-c_r^{\tiny\textcircled{g}}ba_r^{\tiny\textcircled{g}}&c_r^{\tiny\textcircled{g}}
\end{array}
\right),$$ as asserted.\end{proof}

\section{relations with generalized right core-EP inverses}

In this section we investigate relations between generalized right group inverse and generalized right core-EP inverses.
We use $\mathcal{A}^{(1,3)}$ to denote the set of all $(1,3)$-invertible elements in $\mathcal{A}$.

\begin{lem} Let $a\in \mathcal{A}_r^{\#}\bigcap \mathcal{A}^{(1,3)}$. Then $a\in \mathcal{A}_r^{\tiny\textcircled{\#}}$ and
$a_r^{\tiny\textcircled{\#}}=a_r^{\#}aa^{(1,3)}.$\end{lem}
\begin{proof} Let $x=a_r^{\#}$. Then we have
$$ax^2=x=xax, a^2x=axa=a.$$
Set $z=xaa^{(1,3)}$. Then we verify that
$$\begin{array}{rll}
az&=&(axa)a^{(1,3)}=aa^{(1,3)},\\
az^2&=&aa^{(1,3)}xaa^{(1,3)}=aa^{(1,3)}(ax^2)aa^{(1,3)}\\
&=&(aa^{(1,3)}a)x^2aa^{(1,3)}=(ax^2)aa^{(1,3)}=z,\\
(az)^*&=&(aa^{(1,3)})^*=aa^{(1,3)}=az,\\
aza&=&aa^{(1,3)}a=a,
\end{array}$$ thus yielding the result.\end{proof}

\begin{thm} Let $a\in \mathcal{A}$. Then the following are equivalent:\end{thm}
\begin{enumerate}
\item [(1)] $a\in \mathcal{A}_r^{\tiny\textcircled{d}}$.
\vspace{-.5mm}
\item [(2)] $a\in \mathcal{A}_r^{\tiny\textcircled{g}}$ and $a^2a_r^{\tiny\textcircled{g}}\in \mathcal{A}^{(1,3)}$.
\end{enumerate}
In this case, $a_r^{\tiny\textcircled{g}}=(a_r^d)^2aa_r^{\tiny\textcircled{d}}a$ and $a_r^{\tiny\textcircled{d}}=a_r^{\tiny\textcircled{g}}a(a^2a_r^{\tiny\textcircled{g}})^{(1,3)}.$
\begin{proof} $(1)\Rightarrow (2)$ By hypothesis,  we have $$a_r^{\tiny\textcircled{d}}=a(a_r^{\tiny\textcircled{d}})^2, (aa_r^{\tiny\textcircled{d}})^*=aa_r^{\tiny\textcircled{d}}~\mbox{and} ~\lim\limits_{n\to \infty}||a^n-aa_r^{\tiny\textcircled{d}}a^{n}||^{\frac{1}{n}}=0.$$
Set $x=aa_r^{\tiny\textcircled{d}}a$. Then we check that
 $$\begin{array}{rll}
 (aa_r^d)^*(aa_r^d)x&=&(aa_r^d)^*(aa_r^d)aa_r^{\tiny\textcircled{d}}a\\
 &=&(aa_r^d)^*aa_r^{\tiny\textcircled{d}}a\\
 &=&(aa_r^d)^*(aa_r^{\tiny\textcircled{d}})^*a\\
 &=&(aa_r^{\tiny\textcircled{d}}aa_r^d)^*a\\
 &=&(aa_r^d)^*a.
 \end{array}$$ By virtue of Theorem 3.3, we have $$\begin{array}{rll}
 a_r^{\tiny\textcircled{g}}&=&(a_r^d)^2aa_r^dx=(a_r^d)^2aa_r^daa_r^{\tiny\textcircled{d}}a\\
 &=&(a_r^d)^2aa_r^{\tiny\textcircled{d}}a.
 \end{array}$$

Moreover, we check that $$\begin{array}{rll}
[aa_r^{\tiny\textcircled{d}}a]a_r^{\tiny\textcircled{d}}&=&aa_r^{\tiny\textcircled{d}},\\
((aa_r^{\tiny\textcircled{d}}a)a_r^{\tiny\textcircled{d}})^*&=&(aa_r^{\tiny\textcircled{d}}a)a_r^{\tiny\textcircled{d}},\\
(aa_r^{\tiny\textcircled{d}}a)a_r^{\tiny\textcircled{d}}(aa_r^{\tiny\textcircled{d}}a)&=&aa_r^{\tiny\textcircled{d}}a.
\end{array}$$ Therefore $a^2a_r^{\tiny\textcircled{g}}=a^2(a_r^{\tiny\textcircled{d}})^2a=aa_r^{\tiny\textcircled{d}}a\in \mathcal{A}^{(1,3)}.$

$(2)\Rightarrow (1)$ By hypothesis, there exist $x,y\in \mathcal{A}$ such that $$a=x+y, x^*y=yx=0, x\in
\mathcal{A}_r^{\#}, y\in \mathcal{A}^{qnil}$$ and $a_r^{\tiny\textcircled{g}}=x_r^{\#}$. As in the proof Theorem 2.6, we see that
$x=a^2a_r^{\tiny\textcircled{d}}$; hence, $x\in \mathcal{A}^{(1,3)}$. In view of Lemma 4.1, $x\in \mathcal{A}_r^{\tiny\textcircled{\#}}$. Therefore
$a\in \mathcal{A}_r^{\tiny\textcircled{d}}$ by~\cite[Theorem 2.1]{C4}. By using Lemma 4.1 again, we have
$a_r^{\tiny\textcircled{d}}=x_r^{\tiny\textcircled{\#}}=x_r^{\#}xx^{(1,3)}=a_r^{\tiny\textcircled{g}}a(a^2a_r^{\tiny\textcircled{g}})^{(1,3)},$ as asserted.\end{proof}

\begin{cor} Let $a\in \mathcal{A}_r^{\tiny\textcircled{d}}$. Then $$a_r^{\tiny\textcircled{g}}=[a_r^daa_r^{\tiny\textcircled{d}}]^2a.$$\end{cor}
\begin{proof} By virtue of Theorem 4.2, we verify that
$$\begin{array}{rll}
a_r^{\tiny\textcircled{g}}&=&(a_r^d)^2aa_r^{\tiny\textcircled{d}}a=a_r^d(aa_r^{\tiny\textcircled{d}})a_r^daa_r^{\tiny\textcircled{d}}a\\
&=&[a_r^daa_r^{\tiny\textcircled{d}}]^2a,
\end{array}$$ as asserted.\end{proof}

An element $a\in \mathcal{A}$ is partial isometry provided that $a^*=a^{\dag}$ (see~\cite{Z4}). We come now to characterize the generalized right group inverse for partial isometry.

\begin{cor} Let $a\in \mathcal{A}$ be partial isometry. Then $a\in \mathcal{A}_r^{\tiny\textcircled{d}}$ if and only if $a\in \mathcal{A}_r^{\tiny\textcircled{g}}$.\end{cor}
\begin{proof} $\Longrightarrow $ This is obtained by Theorem 4.2.

$\Longleftarrow $ By hypothesis, there exists a generalized right group decomposition $a=a_1+a_2$, e.g., $a_1^*a_2=a_2a_1=0, a_1\in \mathcal{A}_r^{\#}$ and $a_2\in \mathcal{A}^{qnil}$. Since $a\in \mathcal{A}$ is partial isometry, $a^*=a^{\dag}$, i.e., $a=aa^*a$. Then
$$a_1+a_2=a=a(a_1+a_2)^*(a_1+a_2)=a(a_1^*a_1+a_2^*a_2).$$ This implies that
$$(a_1)^2=aa_1^*(a_1)^2=a_1a_1^*(a_1)^2+a_2a_1^*(a_1)^2;$$ hence,
$a_1^*(a_1)^2=a_1^*a_1a_1^*(a_1)^2.$ Since $(a_1)^2(a_1)_r^{\#}=a_1$, we deduce that
$a_1^*a_1=a_1^*a_1a_1^*a_1.$ As the involution is proper, we
have $a_1=a_1a_1^*a_1$. That is, $a_1\in \mathcal{A}$ is partial isometry, and then $a_1^*=a_1^{\dag}$. This implies that
$a_1\in \mathcal{A}^{\dag}$. By using Lemma 4.1, $a_1\in \mathcal{A}_r^{\tiny\textcircled{\#}}$.
According to~\cite[Theorem 2.1]{C4}, $a\in \mathcal{A}_r^{\tiny\textcircled{d}}$, as asserted.\end{proof}

\begin{cor} Let $a\in \mathcal{A}_r^{\tiny\textcircled{d}}$. Then the following are equivalent:\end{cor}
\begin{enumerate}
\item [(1)] $a_r^{\tiny\textcircled{g}}=x$.
\vspace{-.5mm}
\item [(2)] $ax^2=x, ax=a_r^daa_r^{\tiny\textcircled{d}}a.$
\end{enumerate}
\begin{proof} $(1)\Rightarrow (2)$ By virtue of Theorem 4.2, $a\in \mathcal{A}_r^{\tiny\textcircled{g}}$ and $x=a_r^{\tiny\textcircled{g}}=(a_r^d)^2aa_r^{\tiny\textcircled{d}}a$.
Therefore $ax^2=x$ and $ax=a(a_r^d)^2aa_r^{\tiny\textcircled{d}}a=a_r^daa_r^{\tiny\textcircled{d}}a$, as desired.

$(2)\Rightarrow (1)$ By assumption, we have $ax^2=x, ax=a_r^daa_r^{\tiny\textcircled{d}}a.$ Then we verify that
$x=ax^2=(ax)x=a_r^daa_r^{\tiny\textcircled{d}}(ax)=a_r^daa_r^{\tiny\textcircled{d}}(a_r^daa_r^{\tiny\textcircled{d}}a)=
[a_r^daa_r^{\tiny\textcircled{d}}]^2a$. By virtue of Corollary 4.3, $x=a_r^{\tiny\textcircled{g}}$, as required.\end{proof}

\begin{thm} Let $a\in \mathcal{A}_r^{\tiny\textcircled{d}}$. Then $a^3a_r^d\in \mathcal{A}_r^{\tiny\textcircled{\#}}$ and
$$a_r^{\tiny\textcircled{g}}=(aa_r^{\tiny\textcircled{d}}a)_r^{\#}=(a_r^d)^2a^2(a^3a_r^d)_r^{\tiny\textcircled{\#}}a.$$\end{thm}
\begin{proof} We directly check that
$$\begin{array}{rll}
(aa_r^{\tiny\textcircled{d}}a)(a_r^{\tiny\textcircled{g}})^2&=&(aa_r^{\tiny\textcircled{d}}a)[(a_r^d)^2aa_r^{\tiny\textcircled{d}}a]^2=
aa_r^{\tiny\textcircled{d}}(a_r^d)^2aa_r^{\tiny\textcircled{d}}a=a_r^{\tiny\textcircled{g}},\\
(aa_r^{\tiny\textcircled{d}}a)^2a_r^{\tiny\textcircled{g}}&=&(aa_r^{\tiny\textcircled{d}}a)^2(a_r^d)^2aa_r^{\tiny\textcircled{d}}a=
aa_r^daa_r^{\tiny\textcircled{d}}a=aa_r^{\tiny\textcircled{d}}a,\\
(aa_r^{\tiny\textcircled{d}}a)a_r^{\tiny\textcircled{g}}(aa_r^{\tiny\textcircled{d}}a)&=&
(aa_r^{\tiny\textcircled{d}}a)[(a_r^d)^2aa_r^{\tiny\textcircled{d}}a](aa_r^{\tiny\textcircled{d}}a)\\
&=&aa_r^{\tiny\textcircled{d}}[a_r^da^2]a_r^{\tiny\textcircled{d}}a\\
&=&aa_r^{\tiny\textcircled{d}}[aa_r^da]a_r^{\tiny\textcircled{d}}a\\
&=&aa_r^{\tiny\textcircled{d}}a.
\end{array}$$ Then $aa_r^{\tiny\textcircled{d}}a\in \mathcal{A}_r^{\#}$ and
$a_r^{\tiny\textcircled{g}}=(aa_r^{\tiny\textcircled{d}}a)_r^{\#}$

One easily verifies that
$$\begin{array}{rll}
(a^3a_r^d)(a_r^{\tiny\textcircled{d}})^2&=&(a^2(a_r^{\tiny\textcircled{d}})^2=aa_r^{\tiny\textcircled{d}},\\
\big((a^3a_r^d)(a_r^{\tiny\textcircled{d}})^2\big)^*&=&(a^3a_r^d)(a_r^{\tiny\textcircled{d}})^2,\\
(a^3a_r^d)[(a_r^{\tiny\textcircled{d}})^2]^2&=&a(a_r^{\tiny\textcircled{d}})^3=(a_r^{\tiny\textcircled{d}})^2,\\
(a^3a_r^d)[(a_r^{\tiny\textcircled{d}})^2](a^3a_r^d)&=&(aa_r^{\tiny\textcircled{d}})(a^3a_r^d)=a^3a_r^d.
\end{array}$$ Then $a^3a_r^d\in \mathcal{A}_r^{\tiny\textcircled{\#}}$ and $(a^3a_r^d)_r^{\tiny\textcircled{\#}}=(a_r^{\tiny\textcircled{d}})^2.$
In light of Theorem 4.2, we have $$\begin{array}{rll}
a_r^{\tiny\textcircled{g}}&=&(a_r^d)^2aa_r^{\tiny\textcircled{d}}a=(a_r^d)^2a^2[a_r^{\tiny\textcircled{d}}]^2a\\
&=&(a_r^d)^2a^2(a^3a_r^d)_r^{\tiny\textcircled{\#}}a,
\end{array}$$ as required.
\end{proof}

\begin{cor} Let $a\in \mathcal{A}^{\tiny\textcircled{d}}$. Then $a^3a_r^d\in \mathcal{A}^{\tiny\textcircled{\#}}$ and
$$a^{\tiny\textcircled{g}}=(aa^{\tiny\textcircled{d}}a)^{\#}=(a^3a^d)_r^{\tiny\textcircled{\#}}a.$$\end{cor}
\begin{proof} This is obvious by Theorem 4.6.\end{proof}

\begin{lem} Let $a\in \mathcal{A}_r^{\tiny\textcircled{d}}$. Then the following are equivalent:\end{lem}
\begin{enumerate}
\item [(1)] $a_r^{\tiny\textcircled{g}}=x$.
\vspace{-.5mm}
\item [(2)] $im(x)=im(a_r^d)$ and $ax=a_r^daa_r^{\tiny\textcircled{d}}a$.
\vspace{-.5mm}
\item [(3)] $im(x)\subseteq im(a_r^d)$ and $ax=a_r^daa_r^{\tiny\textcircled{d}}a$.
\vspace{-.5mm}
\item [(4)] $x=xax, im(x)=im(a_r^d)$ and $ax=p_{im(a_r^d),ker((a_r^d)^*a)}$.
\end{enumerate}
\begin{proof} $(1)\Rightarrow (2)$ By virtue of Theorem 4.2, $x=(a_r^d)^2aa_r^{\tiny\textcircled{d}}a\in im(a_r^d)$; hence,
$im(x)\subseteq im(a_r^d)$. Construct $a_1$ and $a_2$ as in Theorem 2.6. Then $x=(a_1)_r^{\#}$ and $$a_r^d=(a_1)_r^{\#}+\sum\limits_{n=1}^{\infty}((a_1)_r^{\#})^{n+1}a_2^n.$$ Then
we check that $$xaa_r^d=(a_1)_r^{\#}a(a_1)_r^{\#}+\sum\limits_{n=1}^{\infty}(a_1)_r^{\#}a((a_1)_r^{\#})^{n+1}a_2^n=a_r^d.$$ This implies that $im(a_r^d)\subseteq im(x)$. Hence $im(x)=im(a_r^d)$. In view of Corollary 4.5, $ax=a_r^daa_r^{\tiny\textcircled{d}}a$, as required.

$(2)\Rightarrow (3)$ This is trivial.

$(3)\Rightarrow (1)$ Write $x=a_r^dz$ for a $z\in \mathcal{A}$. Then $$\begin{array}{rll}
ax^2&=&(ax)x=[a_r^daa_r^{\tiny\textcircled{d}}a]x\\
&=&[a_r^daa_r^{\tiny\textcircled{d}}a]a_r^dz=[a_r^da]a_r^dz\\
&=&[a_r^daa_r^d]z=a_r^dz=x.
\end{array}$$ Therefore $a_r^{\tiny\textcircled{g}}=x$ by Corollary 4.5.

$(1)\Rightarrow (4)$ By hypothesis, there exist $z,y\in \mathcal{A}$ such that $$a=z+y, z^*y=yz=0, z\in \mathcal{A}_r^{\#}, ~y\in
\mathcal{A}^{qnil}.$$ Moreover, we have $a_r^{\tiny\textcircled{g}}=z_r^{\#}$.
Hence, $xax=z_r^{\#}(z+y)z_r^{\#}=z_r^{\#}zz_r^{\#}=z_r^{\tiny\textcircled{\#}}=x.$ In view of~\cite[Corollary 2.3]{C4}, $a_r^{\tiny\textcircled{d}}aa_r^{\tiny\textcircled{d}}=a_r^{\tiny\textcircled{d}}$. This implies that $ax\in \mathcal{A}$ is an idempotent.

Claim 1. $im(ax)=im(a_r^d)$. Since $ax=a_r^daa_r^{\tiny\textcircled{d}}a$, we see that
$im(ax)\subseteq im(a_r^d)$. On the other hand, we see that
$$\begin{array}{rll}
   ||aa_r^d-aa_r^{\tiny\textcircled{d}}aa_r^d||^{\frac{1}{n}}&=&||aa^n(a_r^d)^{n+1}-aa_r^{\tiny\textcircled{d}}a^{n+1}(a_r^d)^{n+1}||^{\frac{1}{n}}\\
   &=&||[a^{n+1}-aa_r^{\tiny\textcircled{d}}a^{n+1}](a_r^d)^{n+1}||^{\frac{1}{n}}\\
   &\leq &||a^{n+1}-aa_r^{\tiny\textcircled{d}}a^{n+1}||^{\frac{1}{n}}||a_r^{\tiny\textcircled{d}}||^{1+\frac{1}{n}}.
   \end{array}$$ Hence, $\lim\limits_{n\to \infty}||aa_r^d-aa_r^{\tiny\textcircled{d}}aa_r^d||^{\frac{1}{n}}=0$. Then
   $aa_r^d-aa_r^{\tiny\textcircled{d}}aa_r^d=0$, i.e., $aa_r^{\tiny\textcircled{d}}aa_r^d=aa_r^d$. This implies that $a_r^d=
   a_r^d(aa_r^d)=a_r^d(aa_r^{\tiny\textcircled{d}}aa_r^d)=[a_r^daa_r^{\tiny\textcircled{d}}a]a_r^d=axa_r^d$; whence
  $im(a_r^d)\subseteq im(ax)$. Therefore $im(ax)=im(a_r^d)$.

Claim 2. $ker(ax)=ker((a_r^d)^*a))$. If $ax(r)=0$, by using Corollary 4.5, we have $a_r^daa_r^{\tiny\textcircled{d}}a(r)=0$; hence,
$aa_r^daa_r^{\tiny\textcircled{d}}a(r)=0$. Then
$(aa_r^{\tiny\textcircled{d}})a(r)=0$. This implies that $(aa_r^{\tiny\textcircled{d}})^*a(r)=0$.
Thus $(aa_r^{\tiny\textcircled{d}}a_r^d)^*a(r)=0$. Hence, $(a_r^d)^*a(r)=0$, and so $ker(ax)\subseteq ker((a_r^d)^*a))$. If $(a_r^d)^*a(r)=0$, then $(aa_r^{\tiny\textcircled{d}})^*a(r)=0$. This implies that $(aa_r^{\tiny\textcircled{d}})^*a(r)=0$; whence, $(aa_r^{\tiny\textcircled{d}})a(r)=0$.
Thus $a_r^{\tiny\textcircled{d}}a(r)=0$. By using Corollary 4.5, $ax(r)=0$. Thus $ker(ax)=ker((a_r^d)^*a))$.
Therefore $ax=p_{im(a_r^d),ker((a_r^d)^*a)}$.

$(4)\Rightarrow (1)$ Since $x=xax$, we see that $1-ax\in ker(x)\subseteq ker((a_r^d)^*a))$. Hence,
$(a_r^d)^*a(1-ax)=0$. Then $(aa_r^{\tiny\textcircled{d}})^*a(1-ax)=0$. This implies that
$(aa_r^{\tiny\textcircled{d}})a(1-ax)=0$, and so $a_r^{\tiny\textcircled{d}}a(1-ax)=0$.
Thus $a_r^{\tiny\textcircled{d}}a=a_r^{\tiny\textcircled{d}}a^2x.$ We infer that ??????

Write $x=a_r^dz$ for a $z\in \mathcal{A}$. Then $ax=a_r^daz$; hence,
$$a_r^{\tiny\textcircled{d}}ax=a_r^{\tiny\textcircled{d}}aa_r^dz=a_r^dz=x.$$
By virtue of Corollary 4.5, $a_r^{\tiny\textcircled{g}}=x$, as asserted.\end{proof}

\begin{thm} Let $a\in \mathcal{A}_r^{\tiny\textcircled{d}}$. Then the following are equivalent:\end{thm}
\begin{enumerate}
\item [(1)] $a_r^{\tiny\textcircled{g}}=x$.
\vspace{-.5mm}
\item [(2)] $xax=x, im(x)=im(a_r^d), ker(x)=ker((a_r^d)^*a)$.
\vspace{-.5mm}
\item [(3)] $xax=x, xa=p_{im(a_r^d),ker((a_r^d)^*a^2)}, ker(x)=ker((a_r^d)^*a)$.
\end{enumerate}
\begin{proof} $(1)\Rightarrow (2)$ By virtue of Lemma 4.8, $xax=x, im(x)=im(a_r^d)$.
Let $r\in ker((a^d)^*a)$. Then $(a_r^d)^*a(r)=0$, and so $(aa_r^{\tiny\textcircled{d}})^*a(r)=0$.
This implies that $(aa_r^{\tiny\textcircled{d}})a(r)=0$. Accordingly, we derive $(a_r^{\tiny\textcircled{d}})^2a(r)=0$.
By virtue of Theorem 4.2, $x(r)=0$. Thus $ker((a_r^d)^*a)\subseteq ker(x)$. Let $r\in ker(x)$. Then $x(r)=0$, and so
$(a_r^{\tiny\textcircled{d}})^2a(r)=0$, and then $aa_r^{\tiny\textcircled{d}}a(r)=0$. This implies that
$(aa_r^{\tiny\textcircled{d}})^*a(r)=0$. Thus $(a_r^d)^*a(r)=0$; hence, $r\in ker((a_r^d)^*a)$.
This implies that $ker(x)\subseteq ker((a_r^d)^*a)$. Therefore $ker(x)=ker((a_r^d)^*a)$.

$(2)\Rightarrow (1)$ Since $x=xax$, we see that $ax\in \mathcal{A}$ is an idempotent. Clearly,
$ax\in im(a_r^d)$. Obviously, we have $a_r^d=a(a_r^d)^2\in ax\mathcal{A}$.
Thus, $im(ax)=im(a_r^d)$. If $ax(r)=0$, then $x(r)=xax(r)=0$; hence, $((a_r^d)^*a)(r)=0$. This implies that
$ker(ax)\subseteq ker((a_r^d)^*a)$. If $((a_r^d)^*a)(r)=0$, then $x(r)=0$; and so $ax(r)=0$. This implies that $r\in ker(ax)$.
Hence, $ker(ax)=ker((a_r^d)^*a)$. Therefore $ax=p_{im(a_r^d),ker((a_r^d)^*a)}$. By virtue of Lemma 4.8, $a_r^{\tiny\textcircled{g}}=x$.

$(2)\Rightarrow (3)$ By hypothesis, we have $xax=x, im(x)=im(a_r^d), ker(x)=ker((a_r^d)^*a)$.
Hence, $xa\in \mathcal{A}$ is an idempotent and $im(xa)=im(x)=im(a_r^d).$ If $r\in ker((a_r^d)^*a^2)$, then $ar\in ker(x)=ker((a_r^d)^*a)$. Then
$ar\in ker(x)$. This implies that $r\in ker(xa)$. If $r\in ker(xa)$, then $ar\in ker(x)$. We infer that $ar\in ker((a_r^d)^*a)$; hence,
$(a_r^d)^*a^2r=0$, i.e., $r\in ker((a_r^d)^*a^2)$. Thus $ker(xa)=ker((a_r^d)^*a^2)$. Accordingly, $xa=p_{im(a_r^d),ker((a_r^d)^*a^2)}$.

$(3)\Rightarrow (2)$ As $x=xax$, we get $im(x)=im(xa)=im(a_r^d)$, as required.
\end{proof}

\begin{cor} Let $a\in \mathcal{A}_r^{\tiny\textcircled{d}}$. Then\end{cor}
\begin{enumerate}
\item [(1)] $a_r^{\tiny\textcircled{g}}$ is the unique solution of the system of equations:
$$im(x)=im(a_r^d), ax=p_{im(a_r^d),ker((a_r^d)^*a)}.$$
\item [(2)] $a_r^{\tiny\textcircled{g}}$ is the unique solution of the system of equations:
$$xa=p_{im(a_r^d),ker((a_r^d)^*a^2)}, ker(x)=ker((a_r^d)^*a).$$
\end{enumerate}
\begin{proof} $(1)$ Set $x=a_r^{\tiny\textcircled{g}}$. In view of Theorem 4.9, we have
$$im(x)=im(a_r^d), ax=p_{im(a_r^d),ker((a_r^d)^*a)}.$$

Suppose that there exists $x'\in \mathcal{A}$ such that $$im(x')=im(a_r^d), ax'=p_{im(a_r^d),ker((a_r^d)^*a)}.$$
Then $a(x-x')=0$; hence, $x-x'\in im(a_r^d)\bigcap ker(a)=0$. This implies that $x=x'$, as required.

$(2)$ By virtue of Theorem 4.9, the system of equations is solvable.

Suppose that
$$x_ia=p_{im(a_r^d),ker((a_r^d)^*a^2)}, ker(x_i)=ker((a_r^d)^*a) (i=1,2)$$
Then $x_1a=x_2a$; hence, $x_1aa_r^d=x_2aa_r^d$.

It is easy to verify that $$\begin{array}{rll}
(a_r^d)^*a(1-aa_r^d)&=&(aa_r^{\tiny\textcircled{d}}a_r^d)^*a(1-aa_r^d)\\
&=&(a_r^d)^*(aa_r^{\tiny\textcircled{d}})^*a(1-aa_r^d)\\
&=&(a_r^d)^*(aa_r^{\tiny\textcircled{d}})a(1-aa_r^d)\\
&=&0.
\end{array}$$ Then $1-aa_r^d\in ker(a_r^d)^*a\subseteq ker(x_i)$. This implies that
$x_i=x_iaa_r^d$. Therefore $x_1=x_2$, as desired.\end{proof}

\begin{rem}\end{rem} Let $e\in \mathcal{A}$ be a Hermitian invertiible element. An element $a\in \mathcal{A}$ has generalized right $e$-group inverse if there exists $x\in \mathcal{A}$ such that $x=ax^2, ((ea)^*a^2x)^*=(ea)^*a^2x, \lim\limits_{n\to \infty}||a^n-axa^{n}||^{\frac{1}{n}}=0.$ The preceding $x$ is called the generalized right
$e$-group inverse of $a$, and denoted by $a_r^{e,\tiny\textcircled{g}}$. The generalized right $e$-group inverse can be investigated
using a similar approach. Evidently, the following are equivalent:
\begin{enumerate}
\item [(1)] $a\in \mathcal{A}$ has generalized right $e$-group decomposition, i.e., $$a=a_1+a_2, a_1^*ea_2=a_2a_1=0, a_1\in \mathcal{A}_r^{\tiny\textcircled{\#}}, a_2\in \mathcal{A}^{qnil}.$$
\vspace{-.5mm}
\item [(2)] There exists $x\in \mathcal{A}$ such that $$x=ax^2, ((ea)^*a^2x)^*=(ea)^*a^2x, \lim\limits_{n\to \infty}||a^n-axa^{n}||^{\frac{1}{n}}=0.$$
\vspace{-.5mm}
\item [(3)] $a\in \mathcal{A}_r^{d}$ and there exist $x\in \mathcal{A}$ such that $$x=ax^2, (eaa_r^d)^*a^2x=(eaa_r^d)^*a, \lim\limits_{n\to \infty}||a^n-axa^{n}||^{\frac{1}{n}}=0.$$
\item [(4)]{\it There exists an idempotent $p\in \mathcal{A}$ such that $$(1-p)a(1-p)\in [(1-p)\mathcal{A}(1-p)]_r^{-1}, ((ea)^*ap)^*=(ea)^*ap~\mbox{and} ~pa=pap\in \mathcal{A}^{qnil}.$$}
\item [(5)]{\it $a\in \mathcal{A}_r^{d}$ and there exists an idempotent $p\in \mathcal{A}$ such that $$(1-p)a(1-p)\in [(1-p)\mathcal{A}(1-p)]_r^{-1}, (eaa_r^d)^*ap=0~\mbox{and} ~pa=pap\in \mathcal{A}^{qnil}.$$}
\vspace{-.5mm}
\item [(6)] $a\in \mathcal{A}_r^d$ and there exists an idempotent $q\in \mathcal{A}$ such that $$(aa_r^d)\mathcal{A}=q\mathcal{A}~\mbox{and} ~(ea)^*aq=q^*(ea)^*a.$$
\end{enumerate}

By a similar way, we studied generalized left group inverse. It is proved that $a\in \mathcal{A}^{\tiny\textcircled{g}}$ if and only if $a\in \mathcal{A}_r^{\tiny\textcircled{g}}\bigcap \mathcal{A}_l^{\tiny\textcircled{g}}.$

\end{document}